\documentclass[10pt,a4paper]{article}

\usepackage[T1]{fontenc}
\usepackage{flafter,amsmath,amssymb,latexsym,psfrag,graphicx,color,indentfirst,bm}
\usepackage[amsthm]{ntheorem}
\theorembodyfont{\itshape}
\usepackage{mathrsfs}
\usepackage{bm}
\usepackage{enumerate}
\usepackage{hyperref}
\usepackage{xcolor}
\usepackage{titlesec}
\usepackage{float}
\numberwithin{equation}{section}
\newtheorem{theorem}{Theorem}[section]
\newtheorem{proposition}[theorem]{Proposition}
\newtheorem{lemma}[theorem]{Lemma}
\newtheorem{corollary}[theorem]{Corollary}
\newtheorem{assumption}{Assumption}[section]

\newtheorem{definition}{Definition}
\newtheorem{remark}[theorem]{Remark}

\usepackage[utf8]{inputenc}
\usepackage{dutchcal}

\definecolor{refinecolor}{RGB}{0, 70, 140}

\definecolor{revisionmagenta}{RGB}{180,0,120}
\definecolor{localizationgreen}{RGB}{0,105,80}
\definecolor{rankclarification}{RGB}{0,80,180}
\definecolor{clustercompression}{RGB}{120,35,150}
\definecolor{spacecorrectgreen}{RGB}{0,90,45}
\definecolor{minorrevcolor}{RGB}{0,150,120} 

\hypersetup{
    colorlinks=true,
    linkcolor=refinecolor,
    urlcolor=spacecorrectgreen
}
\usepackage[a4paper,top=2.24cm,bottom=2.24cm,left=2.25cm,right=2.25cm, marginparwidth=1.75cm]{geometry}
\usepackage{tikz}
\usetikzlibrary{arrows.meta, shapes.geometric, patterns, calc}
\newcommand{\R}{\mathbb{R}}

\newcommand{\dd}{\,\mathrm{d}}
\usepackage{comment}

\usepackage{physics}

\newsavebox\foobox
\newlength{\foodim}

\newcommand{\rchi}{\raisebox{2pt}{$\chi$}}

\usepackage{authblk} 

\title{
\Large \textbf{Feedback Stabilization and Profile Tracking for Parabolic Equations}\\ \textbf{ with Singular Thermo-Plasmonic Sources}
}

\author[1]{Arpan Mukherjee}
\author[2]{Sérgio S. Rodrigues\thanks{The work of S. Rodrigues is funded by national funds through the FCT -- Funda\c{c}\~{a}o para a Ci\^{e}ncia e a Tecnologia, I.P., under the scope of the projects UID/00297/2025 (\url{https://doi.org/10.54499/UID/00297/2025}) and UID/PRR/00297/2025 (\url{https://doi.org/10.54499/UID/PRR/00297/2025}) (Center for Mathematics and Applications -- NOVA Math).}}
\author[3]{Mourad Sini\thanks{The work of M. Sini is partially supported by the Austrian Science Fund (FWF): P 32660 and P: 36942.}}

\affil[1]{\small MSU-BIT SMBU Joint Research Center of Applied Mathematics, Shenzhen MSU-BIT University, Shenzhen, People's Republic of China (\texttt{arpan.mukherjee@smbu.edu.cn})}
\affil[2]{\small Center for Mathematics and Applications (NOVA Math) and Department of Mathematics, NOVA School of Science and Technology (NOVA FCT), NOVA University of Lisbon, Campus de Caparica 2829-516, Portugal (\texttt{ssi.rodrigues@fct.unl.pt})}
\affil[3]{\small Johann Radon Institute (RICAM), Austrian Academy of Sciences, Altenbergerstrasse 69, 4040 Linz, Austria (\texttt{mourad.sini@oeaw.ac.at})}

\date{\today}

\usepackage{amssymb}
\begin{document}

\maketitle

\begin{abstract}
    
We present a novel framework for feedback stabilization and profile tracking of the heat equation, driven by distributional actuators derived from a time-domain effective thermo-plasmonic model. The physical temperature field is represented as a superposition of heat kernels with Volterra-type amplitudes, originating from a strongly coupled Maxwell--heat system governing sub-wavelength nanoparticles.

\vspace{0.03in}
\noindent
Formulating the dynamics as an abstract Neumann problem on a bounded domain driven by Dirac sources, we design an output-feedback law utilizing elliptically regularized point evaluations. With $V=H^1(\Omega)$, the stabilization is measured in the energy-dual space $V'$, while the distributional dynamics are understood in $D(A_0)'$, where $A_0$ is the Neumann-Laplacian. We establish exponential decay of the tracking error in $V'$. For non-equilibrium targets, a finite-dimensional fixed-point pre-compensation guarantees exact steady tracking on prescribed low-mode subspaces, yielding explicit tail estimates.

\vspace{0.03in}
\noindent
The core novelty lies in rigorously coupling this infinite-dimensional control problem to a nanoparticle-driven photothermal model. By defining the effective heat generation as a bounded linear actuation map $\mathcal{K}$ from electromagnetic intensities to physical point sources, we prove the total closed-loop tracking error decomposes into a finite-dimensional control projection error and an asymptotic thermo-plasmonic realization remainder. This physical remainder vanishes as the particle scale $\delta \downarrow 0$, while tuning to plasmonic resonances enhances the actuation strength and can improve the conditioning of the right-inverse $\mathcal{K}^\dagger$, thereby reducing the realization error.


\end{abstract}

\newpage

\tableofcontents

\section{Introduction}

\subsection{Thermo-plasmonic heating and effective medium theory}

\noindent 
Metallic (plasmonic) nanoparticles exhibit strong absorption and local field
enhancement when illuminated near their plasmonic resonances. This leads to
highly localized heat generation that can be exploited in a wide range of
applications, including photothermal cancer therapy, thermal imaging,
nano-manipulation and temperature management in micro- and nano-devices; see,
for instance, \cite{BaffouQuidant2013,BaffouNM2020,Chuang2022,
Jorgensen2016,Kim2023, Vines2019}. We discuss in Subsection~\ref{subsec:plasmonic-resonances} how resonant tuning enters the asymptotic modeling and improves the effective actuation strength.

\noindent 
From a modeling viewpoint, quantitative descriptions of light-to-heat
conversion in plasmonic systems typically couple a time-harmonic Maxwell
problem, governing the electromagnetic field, to a heat equation with
volumetric source term representing ohmic losses. In the physical and
engineering literature a variety of models have been proposed to describe the
temperature rise around single particles and suspensions of particles, often
relying on phenomenological or numerical approaches
\cite{BaffouQuidant2013,BaffouNM2020,PhotothermalJAP2016,WuSheldon2023}. On
the mathematical side, asymptotic and effective-medium descriptions have been
developed for plasmonic nanoparticles in the quasi-static and subwavelength
regimes, with a particular emphasis on the role of resonances and high
contrast; see, e.g.,
\cite{AmmariMillienRuizZhang2017,AmmariRomeroRuiz2018,
CaoMukherjeeSini2025,Mukherjee2022} and the references therein.

\noindent 
In \cite{CaoMukherjeeSini2025}, a time-domain effective model for
heat generation by a finite collection of plasmonic nanoparticles in an
otherwise homogeneous background was obtained, under subwavelength and
high-contrast scaling assumptions. The main result shows that the temperature
difference $w$ between the actual temperature and the background solution
without particles can be approximated, outside the particle region, by a
finite superposition of heat kernels centered at the nanoparticle centers,
with time-dependent amplitudes $\sigma_i$ solving a Volterra-type integral
system. The error is quantified in suitable norms, and the amplitudes are
shown to depend linearly (up to higher-order terms) on the internal
electromagnetic energy generated by the incident illumination. This discrete
description is based on a Foldy--Lax point-approximation framework applied to
the underlying Maxwell scattering problem.

\noindent 
The representation of $w$ as a superposition of heat kernels localized at the
particle centers is the key structural feature that we exploit here. It
suggests that, at least in the effective description, a suitably designed
ensemble of plasmonic nanoparticles can be interpreted as a system of point
actuators for the heat equation, remotely driven by electromagnetic inputs.

\subsection{Heat-profile tracking and point actuators}

\noindent  
Independently of the thermo--plasmonic context, the control and stabilization
of linear parabolic equations has been extensively studied; see, e.g.,
\cite{Barbu2018,KrsticSmyshlyaev2008,Pazy1983,RaymondBook,TucsnakWeiss2009}
and the references therein. A particularly relevant class of problems for the
present work involves finite-dimensional actuation and observation, where a
heat equation is driven by a finite number of internal or boundary actuators
and observed through finitely many output functionals.

\noindent 
Pointwise (Dirac-type) actuation and observation models are natural
idealizations when the actual sources and sensors are localized on spatial
scales much smaller than the size of the domain. In the context of the heat
equation, controllability and stabilization issues for point actuators and for
controls acting on small or moving subsets have been investigated in various
settings; see, for instance, \cite{CasZuazua, FardigolaKhalina2022,Hamidoglu2013,Letrouit2019, Martinez2015,Rodrigues2021} and the references
therein. The analysis of such systems typically relies on semigroup methods,
spectral decompositions and Carleman estimates.
\bigskip

\noindent 
In the present paper, we consider a controlled heat equation on a bounded
domain $\Omega\subset\R^d$, mainly $d=3$, with homogeneous Neumann boundary conditions, driven by finitely many point actuators. We design an output feedback based on
\emph{elliptically regularized} point measurements, namely pointwise
evaluations of $\mathcal A^{-1}y$, where $\mathcal A=I-A_0$ and $A_0$ is the
Neumann diffusion operator. For a prescribed reference profile $y_r$, the
feedback has the form
\[
u_j(t) = (u_r)_j-\lambda\Big[(\mathcal A^{-1}y)(z_j,t) - (\mathcal A^{-1}y_r)(z_j)\Big],
\qquad j=1,\dots,M,
\]
where $\lambda>0$ is a gain, $z_j$ are the point actuator locations, and $u_r\in\R^M$ is a constant feedforward used to
compensate the fact that non-constant references are not equilibria of the
free Neumann heat dynamics (and which is further refined via a fixed-point
pre-compensation on low modes; see Section~\ref{sec:problem-formulation}).
Working in the $V'$ framework, \color{black}{with $V:=H^1(\Omega)$ and with the Dirac actuators acting in $D(\mathcal A)'$,} we use the stabilization theory in \cite{KunRodWal24-cocv} to obtain exponential convergence of the tracking
error in $V'$ for this class of feedbacks.

\subsection{Objective and main contributions}

\noindent
The goal of this work is to connect rigorously the discrete thermo-plasmonic
effective model of \cite{CaoMukherjeeSini2025} with heat-profile tracking
feedback laws for a heat equation with point actuators. More precisely, we aim
at:
\begin{itemize}
\item[(i)] Recasting the discrete thermo-plasmonic model as a heat equation
on a bounded domain driven by finitely many point actuators, in a
functional-analytic setting suitable for control design. In particular, we
identify the effective point-source amplitudes in terms of the Volterra system
for the nanoparticle amplitudes and provide a quantitative justification for
restricting the whole-space heat propagation in $\mathbb{R}^3$ to a bounded Neumann
control domain on compact subsets away from $\partial\Omega$.
\item[(ii)] Establishing well-posedness and basic stability properties of a
feedback law for heat-profile tracking based on pointwise evaluations of
$\mathcal A^{-1}y$, interpreted as a closed-loop parabolic system with point
actuators. We place the closed loop in the distribution-actuator framework of
\cite{KunRodWal24-cocv} in the state space $V'$ and obtain exponential
convergence to an equilibrium. For non-equilibrium target profiles we
incorporate a feedforward term and a fixed-point pre-compensation on a low-mode
space $X_N$, yielding exact matching of the first $N$ Neumann modes of the
steady reached profile and an explicit tail estimate.
\item[(iii)] Coupling the effective thermo-plasmonic model with the feedback
design and quantifying the difference between the ideal closed-loop
temperature (in the abstract control model) and the temperature realized
physically through plasmonic actuation. At this level we obtain a tracking
bound that separates a projection error (restriction of the abstract feedback
to the finite-dimensional reachable range of the complete Volterra actuation
map) from a conditioned thermo-plasmonic realization remainder.  We retain the
smallest nonzero singular value of the calibrated map explicitly and explain
how plasmonic resonances can improve control authority and conditioning.
\end{itemize}

\noindent
The main message is that---under the scaling and smallness assumptions
guaranteeing the validity of the discrete thermo-plasmonic model, and under
standard controllability-type conditions on the actuator locations and on the
illumination patterns---it is meaningful to interpret a configuration of
plasmonic nanoparticles as a system of actuators for the heat equation. By
driving the nanoparticles with suitably designed electromagnetic fields, one
can approximate abstract feedback laws and hence track a desired heat profile. At the level of steady targets, the strategy yields exact matching of the desired profile in a prescribed low-dimensional spectral subspace and an explicit estimate of the remaining tail mismatch. A schematic overview of this coupled computational and physical architecture is provided in Figure~\ref{fig:schematic}.

\begin{figure}[htbp]
    \centering
    \begin{tikzpicture}[
        >=Latex,
        font=\small,
        cyberbox/.style={rectangle, draw=sciBlue!80!black, fill=sciBlue!8, thick, rounded corners=4pt, align=center, minimum height=1.6cm, inner sep=6pt},
        physbox/.style={rectangle, draw=sciOrange!80!black, fill=sciOrange!8, thick, rounded corners=4pt, align=center, minimum height=1.6cm, inner sep=6pt},
        sensorbox/.style={rectangle, draw=sciSlate!90!black, fill=sciSlate!10, thick, rounded corners=4pt, align=center, minimum height=1.3cm, inner sep=6pt},
        arrow/.style={draw, ->, thick, color=black!85},
        labelbox/.style={fill=white, inner sep=2.5pt, font=\footnotesize, draw=black!15, rounded corners=2pt}
    ]
        \definecolor{sciBlue}{RGB}{0, 114, 189}    
        \definecolor{sciOrange}{RGB}{217, 83, 25}  
        \definecolor{sciSlate}{RGB}{119, 172, 48}  

        \draw [draw=black!40, dashed, fill=black!3, rounded corners=6pt] (-1.5, 0.0) rectangle (16.4, 3.1);
        \node [anchor=north west, font=\bfseries\color{sciBlue!90!black}] at (1.4, 3) {Abstract Control};

        \draw [draw=black!40, dashed, fill=black!3, rounded corners=6pt] (-1.5, -4.4) rectangle (16.4, -1.5);
        \node [anchor=north west, font=\bfseries\color{sciOrange!90!black}] at (1.4, -1.5) {Thermo-Plasmonic Realization (Physical)};

        \node [cyberbox, text width=2.2cm] (target) at (0.0, 1.5) {\textbf{Target Profile}\\[4pt]$y_r \in X_N$};
        \node [cyberbox, text width=3.6cm] (precomp) at (4.6, 1.5) {\textbf{Bias Pre-comp}\\[4pt]$y^\star = (I+T_N)^{-1}y_r$\\[3pt]$u_r = U_N y^\star$};
        \node [cyberbox, text width=3.6cm] (controller) at (9.8, 1.5) {\textbf{Feedback Law}\\[4pt]$z = y_{\mathrm{phys}} - y^\star$\\[3pt]$u_{\mathrm{id}} = u_r - \lambda C z$};
        \node [cyberbox, text width=3.0cm] (project) at (14.6, 1.5) {\textbf{Finite-Dim. Proj}\\[4pt]$u_{\mathrm{des}} = P_{Y_0} u_{\mathrm{id}}$};

        \node [physbox, text width=3.0cm] (inverse) at (14.6, -3.1) {\textbf{EM Inverse Map}\\[4pt]$p=\mathcal K_{\delta,Q}^{\dagger}u_{\mathrm{des}}$};
        \node [physbox, text width=4.4cm] (dynamics) at (8.5, -3.1) {\textbf{Thermo-Plasmonic Dynamics}\\[4pt]$G = \mathcal{K}_\delta p + \mathcal{R}^\delta p$\\[3pt]$\dot{y}_{\mathrm{phys}} = A_0 y_{\mathrm{phys}} + G$};
        \node [physbox, text width=2.6cm] (output) at (1.5, -3.1) {\textbf{Physical Temp}\\[4pt]$y_{\mathrm{phys}}(t) \in V'$};

        \node [sensorbox, text width=3.6cm] (sensor) at (5.6, -0.75) {\textbf{Elliptic Observation}\\[4pt]$Cy = (\mathcal{A}^{-1}y)(z_j)$};

        \draw [arrow] ([yshift=6pt]target.east) -- node[above, labelbox, yshift=3pt] {$y_r$} ([yshift=6pt]precomp.west);
        \draw [arrow] ([yshift=6pt]precomp.east) -- node[above, labelbox, yshift=3pt] {$y^\star, u_r$} ([yshift=6pt]controller.west);
        \draw [arrow] ([yshift=6pt]controller.east) -- node[above, labelbox, yshift=3pt] {$u_{\mathrm{id}} \in L^2$} ([yshift=6pt]project.west);
        
        \draw [arrow] (project.south) -- node[right, labelbox, pos=0.5] {$u_{\mathrm{des}} \in Y_0$} (inverse.north);
        
        \draw [arrow] ([yshift=6pt]inverse.west) -- node[above, labelbox, yshift=3pt] {$p\in U_0$} ([yshift=6pt]dynamics.east);
        \draw [arrow] ([yshift=6pt]dynamics.west) -- node[above, labelbox, yshift=3pt] {$G \in Y$} ([yshift=6pt]output.east);
        
        \draw [arrow] (output.north) |- (sensor.west);
        \draw [arrow] (sensor.east) -| node[right, labelbox, pos=0.75] {$C y_{\mathrm{phys}}$} (controller.south);

    \end{tikzpicture}
    \caption{Schematic block diagram of the thermo-plasmonic tracking framework. The top computational layer generates ideal Dirac inputs via finite-dimensional fixed-point pre-compensation and $V'$-regularized output feedback. The bottom physical layer projects these inputs onto a realizable subspace $Y_0$ and drives the effective physical temperature $y_{\mathrm{phys}}$ via electromagnetic intensities $p(t)$. The feedback loop closes across the layers via the elliptically regularized observation operator $C$.}
    \label{fig:schematic}
\end{figure}

\section{Problem Formulation and Main Results}\label{sec:problem-formulation}

\subsection{Preliminary Functional Framework} \label{subsec:functional-framework}

\noindent
Let $\Omega \subset \mathbb{R}^3$ be a bounded, connected $C^2$ domain. The state space is $H := L^2(\Omega)$ with inner product $\langle \cdot, \cdot \rangle_H$. We define the Neumann diffusion operator $A_0 := \kappa_m \Delta$ with dense domain\footnote{While we focus on Neumann conditions to model an insulated control volume, the abstract stabilization framework can also be stated using homogeneous Dirichlet boundary conditions by taking $D(A_0) := H^2(\Omega) \cap H_0^1(\Omega)$.}
\begin{equation}
    D(A_0) := \{ y \in H^2(\Omega) : \partial_\nu y = 0 \text{ on } \partial\Omega \}.
\end{equation}
Since $A_0$ is self-adjoint and dissipative with compact resolvent, $H$ admits an orthonormal eigenbasis $\{\varphi_k\}_{k=1}^\infty$ with eigenvalues $0 = \lambda_1 \le \lambda_2 \le \dots \to \infty$ satisfying $-A_0 \varphi_k = \lambda_k \varphi_k$. 

\medskip
\noindent
We define the coercive elliptic operator $\mathcal{A} := I - A_0$. \color{black}{In the Hilbert scale used for stabilization we set $V:=H^1(\Omega)=D(\mathcal A^{1/2})$ and $V':=H^1(\Omega)'$, while the operator domain remains $D(\mathcal A)=D(A_0)$. Thus $D(\mathcal A)\hookrightarrow V\hookrightarrow H\hookrightarrow V'\hookrightarrow D(\mathcal A)'$.}

\medskip
\noindent
\color{black}{For $d=3$, $D(\mathcal A)\hookrightarrow H^2(\Omega)\hookrightarrow C(\overline\Omega)$, so point evaluation is continuous on $D(\mathcal A)$ and $\delta_{z_j}\in D(\mathcal A)'$.} The control operator $B: \mathbb{R}^M \to \color{black}{D(\mathcal A)'}$ is defined by
\begin{equation}
    B u := \sum_{j=1}^M u_j \delta_{z_j}.
\end{equation}

\medskip
\noindent
\textbf{Spectral Truncation and Control Spaces.} \\
For $N \in \mathbb{N}$, we define the low-mode subspace $X_N := \mathrm{span}\{\varphi_1, \dots, \varphi_N\} \subset H$, with $H$-orthogonal projection $P_N: H \to X_N$ and high-frequency tail $X_N^\perp := (I - P_N)H$.
\medskip

\noindent
For $T>0$, the maximal regularity space for the evolution equation is
\begin{equation}
    \color{black}{W(0,T; H, D(\mathcal A)') := \big\{ y \in L^2(0,T;H) : \dot{y} \in L^2(0,T;D(\mathcal A)') \big\}.}
\end{equation}
\color{black}{Moreover, $W(0,T;H,D(\mathcal A)')\hookrightarrow C([0,T];V')$ by the standard Hilbert-scale trace embedding.}
The temporal spaces governing the thermo-plasmonic control mapping are defined as follows
\begin{enumerate}
    \item [(i)]  For $k \in \mathbb{N}$,  set
    $ H^k(0,T; \mathbb{R}^M) := \left\{ v \in L^2(0,T; \mathbb{R}^M) \;\middle|\; \sum_{j=0}^k \int_0^T |\partial_t^j v(t)|^2 \dd t < \infty \right\}$,
    where $\partial_t^j$ denotes the weak derivative. In particular, we define the state spaces for the effective heat inputs and the point-source temporal amplitudes as $Y := L^2(0,T; \mathbb{R}^M)$ and $X := H^1(0,T; \mathbb{R}^M)$, respectively.
    
    \item[(ii)] $U := H^{r+1/2}_{0,\sigma}(0,T; \mathbb{R}^P) := \Big\{ p : (0,T) \to \mathbb{R}^P \;\Big|\; e^{-\sigma t}p \in H^{r+1/2}(0,T; \mathbb{R}^P) \text{ and } \partial_t^n p(0) = 0 \ (\forall 0 \le n \le r) \Big\}$: Exponentially weighted fractional Sobolev space of causal intensity controls.
\end{enumerate}
In Section~\ref{sec:K-inverse}, we choose a finite-dimensional illumination space
$U_0\subset U$ from a prescribed dictionary of causal temporal profiles and electromagnetic
illumination channels.  For the fixed nanoparticle configuration of scale $\delta$, the associated
realizable effective-input space is defined by
\[
   Y_0=Y_0^\delta:=\mathcal K_\delta(U_0)\subset Y,
\]
where $\mathcal K_\delta$ is the complete intensity-to-heat map, including the Volterra dynamics.
Thus the output basis consists of the actual time histories generated by the chosen illuminations;
it is not assumed to share the temporal profiles used to define $U_0$.  The superscript $\delta$ is
suppressed below whenever the configuration is fixed.

\subsection{Model Formulation: From Dirac Actuation to Thermo-Plasmonic Coupling}
\label{subsec:math-formulation}

\noindent
This section formalizes the heat-profile tracking problem by coupling an abstract Dirac-actuated parabolic system with its physical thermo-plasmonic realization. We state theoretical guarantees for exponential stabilization in the extrapolation space $V'$. By exploiting finite-dimensional spectral truncation and the bounded invertibility of the effective actuation map, we quantitatively link the abstract tracking dynamics to localized electromagnetic heating governed by time-domain Volterra systems.

\subsubsection{The Parabolic System with Point Actuators}

\noindent
Let $\Omega \subset \mathbb{R}^3$ be a bounded, connected, $C^2$ domain. The restriction of the background heat equation to $\Omega$ (subject to homogeneous Neumann boundary conditions) is given by
\begin{equation}
\label{eq:3.1}
\begin{cases}
\partial_t y - \kappa_m\Delta y
= \displaystyle\sum_{j=1}^M u_j(t)\,\delta_{z_j}
& \text{in }\Omega \times (0,T),\\[0.3em]
\partial_\nu y = 0 & \text{on }\partial\Omega \times (0,T),\\[0.3em]
y(\cdot,0) = y_0 & \text{in }\Omega,
\end{cases}
\end{equation}
where $\kappa_m > 0$ denotes the thermal diffusivity of the background medium, and $u_j(t)$ represents the time-dependent control inputs applied at the fixed actuator locations $z_j \in \Omega$.

\subsubsection{The Coupled Transmission Problem and its Effective Limit} \label{sec:model}

\noindent
We formulate the thermo-plasmonic transmission problem and its effective discrete description, relying on the asymptotic framework established in \cite{CaoMukherjeeSini2025}.

\paragraph{Geometric Configuration.}
The physical space is $\mathbb{R}^3$, consisting of a homogeneous background medium and a finite, dense collection of $M$ disjoint, resonant inclusions (modeling the electromagnetic nanoparticles \textcolor{black}{which serve as the $M$ point actuators}). Let $\delta \in (0, \delta_0)$ with $\delta_0 \ll 1$ denote the scale parameter. The inclusion region $D^\delta \subset \mathbb{R}^3$ is defined as the union of $M$ disjoint, $\mathcal{C}^2$-regular domains
\begin{equation}
    D^\delta := \bigcup_{m=1}^M D^\delta_m,
\end{equation}
where each $D^\delta_m$ is generated by an affine diffeomorphism $\mathcal{T}^\delta_m: \mathbb{R}^3 \to \mathbb{R}^3$, defined by $\mathcal{T}^\delta_m(x) := \delta x + z_m$, applied to a reference domain $B_m$:
\begin{align}
    D^\delta_m := \mathcal{T}^\delta_m(B_m), \quad m \in \{1, \dots, M\}.
\end{align}
To ensure the inclusions are mutually disjoint, the distribution of centers $\mathcal{Z}_M := \{z_m\}_{m=1}^M \subset \mathbb{R}^3$ satisfies a strict separation condition. We define the strictly positive particle separation distance $d_{\rm sep}$ and the maximum characteristic diameter \textcolor{black}{$d_{\max}$} as:
\begin{equation}
    d_{\rm sep} := \min_{1 \leq i < j \leq M} \operatorname{dist}(D^\delta_i, D^\delta_j) > 0, \quad \textcolor{black}{d_{\max}} := \max_{1 \leq m \leq M} \operatorname{diam}(D^\delta_m) = \mathcal{O}(\delta).
\end{equation}
The homogeneous, isotropic background medium is denoted by $\mathcal{X}_\delta := \mathbb{R}^3 \setminus \overline{D^\delta}$.

\paragraph{Coupled Maxwell-Heat Dynamics.}
For $t \in (0, T)$, the system is governed by a coupled parabolic-Maxwell transmission problem across the interface $\Gamma^\delta := \bigcup_{m=1}^M \partial D^\delta_m$ with outward unit normal $\nu$. Assuming a globally nonmagnetic medium ($\mu_\delta \equiv \mu_m$, where $\mu$ denotes the magnetic permeability), the piecewise constant scalar constitutive parameters are defined using the indicator function $\rchi$:
\begin{align*}
q_\delta(x) &= q_m \rchi_{\mathcal{X}_\delta}(x) + \sum_{m=1}^M q_{D_m} \rchi_{D^\delta_m}(x), \quad q \in \{c, \gamma\}, \\
\varepsilon_\delta(x, \omega) &= \varepsilon_m \rchi_{\mathcal{X}_\delta}(x) + \sum_{m=1}^M \varepsilon_{D_m}(\omega) \rchi_{D^\delta_m}(x).
\end{align*}
Here, the parameter $q \in \{c, \gamma\}$ denotes the respective constitutive property, where $c$ is the volumetric heat capacity and $\gamma$ is the thermal conductivity. Consequently, $c_{D_m}$ and $\gamma_{D_m}$ denote the heat capacity and thermal conductivity inside the inclusion $D^\delta_m$, while $c_m$ and $\gamma_m$ represent the corresponding properties in the background. The complex electric permittivity is denoted by $\varepsilon$.

\noindent
The transient thermal potential $u(x,t)$ satisfies the isotropic parabolic transmission system:
\begin{equation}
\begin{cases} 
c_{D_m} \partial_t u - \gamma_{D_m} \Delta u = \frac{\omega}{2\pi} \Im(\varepsilon_{D_m}) |\mathbf{E}|^2 f(t) & \text{in } D^\delta_m \times (0,T), \\ 
c_m \partial_t u - \gamma_m \Delta u = 0 & \text{in } \mathcal{X}_\delta \times (0,T), \\ 
\left[ u \right] = 0, \quad \left[ \gamma_\delta \partial_\nu u \right] = 0 & \text{on } \Gamma^\delta \times (0,T), \\ 
u(\cdot,0) = 0 & \text{in } \mathbb{R}^3, 
\end{cases}
\end{equation}
where $[\cdot]$ denotes the trace jump across $\Gamma^\delta$. The causal source modulation $f \in \mathcal{C}^\infty([0,T])$ satisfies $\partial_t^{r+1} f(0) \neq 0$ and $\partial_t^n f(0) = 0$ for $0 \leq n \leq r$. The background is lossless ($\Im(\varepsilon_m) = 0$).

\medskip
\noindent
The total electric field $\mathbf{E}$ satisfies the time-harmonic Maxwell transmission problem at driving frequency $\omega$:
\begin{equation}
\begin{cases}
\mu_m^{-1} \nabla \times \nabla \times \mathbf{E} - \omega^2 \varepsilon_m \mathbf{E} = 0 & \text{in } \mathcal{X}_\delta, \\
\mu_m^{-1} \nabla \times \nabla \times \mathbf{E} - \omega^2 \varepsilon_{D_m} \mathbf{E} = 0 & \text{in } D^\delta_m, \\
\left[ \nu \times \mathbf{E} \right] = 0, \quad \left[ \nu \times \left( \nabla \times \mathbf{E} \right) \right] = 0 & \text{on } \Gamma^\delta.
\end{cases}
\end{equation}
The scattered field $\mathbf{E}^{\mathrm{sc}} := \mathbf{E} - \mathbf{E}^{\mathrm{in}}$ satisfies the Silver-M\"uller radiation condition:
\begin{equation}
\lim_{|x| \to \infty} |x| \left( \nabla \times \mathbf{E}^{\mathrm{sc}} \times \frac{x}{|x|} - i k_m \mathbf{E}^{\mathrm{sc}} \right) = 0,
\end{equation}
where $k_m := \omega \sqrt{\varepsilon_m \mu_m}$ is the background wavenumber. Within $D^\delta_m$, the complex permittivity is governed by the Lorentz oscillator model:
\begin{equation}\label{mod:lorentz}
\varepsilon_{D_m}(\omega) = \varepsilon_\infty \left( 1 + \frac{\omega_p^2}{\omega_0^2 - \omega^2 - i \zeta \omega} \right),
\end{equation}
parameterized by the high-frequency permittivity $\varepsilon_\infty$, plasma frequency $\omega_p$, resonance frequency $\omega_0$, and damping constant $\zeta > 0$.

\medskip
\noindent
To analyze the electromagnetic resonances within the inclusion geometries $B_m$, we recall the Helmholtz decomposition (see \cite{D-L, raveski} for details):
\begin{align} \label{decomposition-introduction}
\big(L^2(B_m)\big)^3 = \mathbb{H}_{0}(\operatorname{div} 0, B_m) \oplus \mathbb{H}_{0}(\operatorname{curl} 0, B_m) \oplus \nabla \mathbb{H}_{\mathrm{arm}},
\end{align}
where
\begin{align}
    \begin{cases}
        \mathbb{H}_{0}(\operatorname{div} 0, B_m) = \big\{ v \in \mathbb{H}(\operatorname{div}, B_m): \operatorname{div} v = 0 \text{ in } B_m, \; v \cdot \nu = 0 \text{ on } \partial B_m \big\}, \\
        \mathbb{H}_{0}(\operatorname{curl} 0, B_m) = \big\{ v \in \mathbb{H}(\operatorname{curl}, B_m): \operatorname{curl} v = 0 \text{ in } B_m, \; v \times \nu = 0 \text{ on } \partial B_m \big\}, \\
        \nabla \mathbb{H}_{\mathrm{arm}} = \big\{ v \in \big(L^2(B_m)\big)^3: \exists \varphi \in H^1(B_m) \text{ with } \Delta \varphi = 0 \text{ and } v = \nabla \varphi \big\}.
    \end{cases}
\end{align}
The quasi-static magnetization operator associated with the reference domain $B_m$ is defined by
\[
\mathbb{M}^{(0)}_{B_m}[v](x) := \nabla \int_{B_m} \nabla_y \frac{1}{4\pi |x - y|} \cdot v(y)  dy.
\]
It is well known (see, for instance, \cite{friedmanI}) that \(\mathbb{M}^{(0)}_{B_m}: \nabla \mathbb{H}_{\operatorname{arm}} \rightarrow \nabla \mathbb{H}_{\operatorname{arm}}\) admits a complete orthonormal basis of eigenfunctions, denoted by \(\big(\lambda^{(3)}_{n,m}, e^{(3)}_{n,m}\big)_{n \in \mathbb{N}}\). 

\paragraph{Discrete Effective Thermo-Plasmonic Model.}

To formulate the effective description, we adopt the asymptotic scaling regime established in \cite{CaoMukherjeeSini2025}. For a characteristic scale parameter $\delta \ll 1$, the nanoparticle damping $\zeta$, inner thermal conductivity $\gamma_p$, and inner volumetric heat capacity $c_p$ satisfy
\begin{equation}\label{scales}
    \zeta \sim \delta^h, \quad \gamma_p \sim \delta^{-\beta}, \quad \text{and} \quad c_p \sim 1,
\end{equation}
where the scaling exponents satisfy $\beta \in (0,2)$ and $\frac{9}{5} < h < 2$. Via the Drude model (considering $\omega_0=0$ in \eqref{mod:lorentz}) at moderate incident frequencies $k$, this implies $\Im(\varepsilon_p) \sim \delta^h$. The background parameters $c_m$, $\gamma_m$, and $\varepsilon_m$ remain of order $\mathcal{O}(1)$.
\newline
Let $w := u - u^{(0)}$ denote the temperature variation relative to the homogeneous background solution $u^{(0)}$. Throughout the control part of this paper we use
\[
    \kappa_m:=\frac{\gamma_m}{c_m}
\]
for the background thermal diffusivity. Equivalently, the notation inherited from \cite{CaoMukherjeeSini2025} uses the inverse diffusivity
$\widehat{\kappa}_m:=c_m/\gamma_m=\kappa_m^{-1}$.  In the present notation the whole-space heat kernel is
\begin{equation}
    \Phi^{(m)}(x, t; y, \tau) = \frac{1}{\big(4\pi\kappa_m(t-\tau)\big)^{3/2}}
    \exp\left( -\frac{|x-y|^2}{4\kappa_m(t-\tau)} \right) \rchi_{(0, \infty)}(t - \tau).
\end{equation}
Tuning the incident frequency to a sub-wavelength plasmonic resonance ($k\delta \ll 1$), the quasi-static magnetization operator satisfies the spectral condition
\begin{equation}
    1 + \eta \lambda^{(3)}_{n_0,m} = \pm c_0 \delta^h,
\end{equation}
for a constant $c_0 > 0$, where $\eta := \varepsilon_p - \varepsilon_m$ is the electromagnetic contrast, and $|1 + \eta \lambda^{(3)}_{n,m}| \sim 1$ for $n \neq n_0$.
\newline
To guarantee the bounded invertibility of the governing thermo-plasmonic systems, we impose the following geometric and spectral conditions
\begin{align}
    b \, C_{\sigma_0,\widehat{\kappa}_m} \max_{1 \leq i \leq M} \sum_{j \neq i} d_{ij}^{-2} &< 1, \quad \text{where } b := \max_{1 \leq j \leq M} \frac{\alpha_j}{c_m} \operatorname{vol}(D^\delta_j) \sim \delta^{3-\beta}, \label{cond:invert_heat} \\
    |\eta| \max_{1 \leq i \leq M} \|\mathcal{P}_{D_i}\| d_{\rm sep}^{-3} &< 1, \quad \text{and} \quad \frac{\omega}{4\pi c_0} < 1, \label{cond:invert_maxwell}
\end{align}
where $d_{ij} := |z_i - z_j|$, $d_{\rm sep} := \min_{i \neq j} d_{ij}$, $\alpha_j := \gamma_{p,j} - \gamma_m$, $C_{\sigma_0,\widehat{\kappa}_m} > 0$ is a constant depending only on the Laplace parameter $\sigma_0$ and the inverse diffusivity $\widehat{\kappa}_m$, and $\mathcal{P}_{D_i} = \mathcal{O}(\delta^{3-h})$ is the effective polarization matrix of the $i$-th inclusion.

\begin{theorem}[Effective Thermo-Plasmonic Dynamics \cite{CaoMukherjeeSini2025}]
\label{thm:effective_dynamics}
Assume $f \in H^{r-1/2}_{0,\sigma}(0,T)$ is a causal function and conditions \eqref{cond:invert_heat}--\eqref{cond:invert_maxwell} hold. Then, the underlying governing heat equation admits a unique weak solution $w \in H^r_{0,\sigma}\big(0, T; H^1(\mathbb{R}^3)\big)$. Let $\mathcal{K} \subset \mathbb{R}^3$ be a bounded domain such that $\overline{D^\delta} \subset\subset \mathcal{K}$ with $\operatorname{dist}(\partial \mathcal{K}, \overline{D^\delta}) > c > 0$. As $\delta \to 0$, the temperature variation $w(x,t)$ satisfies for $(x, t) \in (\mathbb{R}^3 \setminus \mathcal{K}) \times (0, T)$:
\begin{equation}
    w(x, t) = -\sum_{i=1}^M \frac{\alpha_i}{c_m} \operatorname{vol}(D^\delta_i) \int_0^t \Phi^{(m)}(x, t; z_i, \tau) \sigma^{(i)}(\tau) \dd\tau + \mathcal{O}(M \delta^{4 - h}).
\end{equation}
The temporal amplitudes $\sigma^{(i)} \in H^1(0,T)$ uniquely solve the Volterra system:
\begin{equation}
    \sigma^{(i)}(t) + \sum_{\substack{j=1 \\ j \neq i}}^M \frac{\alpha_j}{c_m} \operatorname{vol}(D^\delta_j) \int_0^t \Phi^{(m)}(z_i, t; z_j, \tau) \frac{\partial}{\partial \tau} \sigma^{(j)}(\tau) \dd\tau = F_i(t),
\end{equation}
where the photothermal forcing term is given by
\begin{equation}
    F_i(t) = \frac{\omega \Im(\varepsilon_p)}{2 \pi \gamma_{p,i}} \, \delta^{\beta - h} \, f(t) \left( \overline{\mathcal{Q}}_i^\top \mathcal{P}_{B} \mathcal{Q}_i \right),
\end{equation}
with $\mathcal{P}_{B}$ being the normalized reference polarization tensor. The field vectors $\mathcal{Q}_i \in \mathbb{C}^3$ satisfy the algebraic Foldy-Lax system:
\begin{equation}
    \mathcal{Q}_i - \eta \sum_{\substack{j=1 \\ j \neq i}}^M \bm{\Upsilon}^{(\omega)}(z_i, z_j) \mathcal{P}_{D_j} \mathcal{Q}_j = \mathbf{E}^{\mathrm{in}}(z_i),
\end{equation}
where $\bm{\Upsilon}^{(\omega)}(x,y) := \operatorname{Hess}_x \mathcal{G}^{(\omega)}(x,y) + \omega^2 \mathcal{G}^{(\omega)}(x,y) \mathbb{I}$ is the dyadic Green's function.
\end{theorem}

\subsubsection{Plasmonic Resonances and Actuation Strength} \label{subsec:plasmonic-resonances}

\noindent
In the control reduction developed below, the restriction of the complete intensity-to-heat
operator to $U_0$ is represented by the finite-dimensional realization operator
$\mathbf K_{\delta,Q}$, whose range is $Y_0$.  Its control authority is measured by the smallest
nonzero singular value $c_{\rm act}(\delta,Q)$, or equivalently by the inverse-amplification factor
$C_{\rm inv}(\delta,Q)=\|\mathcal K_{\delta,Q}^{\dagger}\|$.  The resonant scaling
$\mathcal P_{D_i}\sim \delta^{3-h}$ increases the gain of each actuator channel and can therefore:
\begin{enumerate}
\item[(i)] increase $c_{\rm act}(\delta,Q)$ and reduce the amplification in the right-inverse needed
to realize a prescribed $u_{\mathrm{des}}\in Y_0$,
\item[(ii)] improve robustness with respect to modeling and measurement errors, and
\item[(iii)] allow one to enrich the realizable control space $Y_0$ (hence reduce projection error)
for a fixed admissible illumination energy.
\end{enumerate}
Thus, plasmonic resonances are not only the physical origin of strong local heating, but also a
mathematical lever to strengthen actuation and improve tracking accuracy.  Full rank guarantees a
right-inverse for each fixed configuration; uniform asymptotic estimates additionally require
quantitative control of $c_{\rm act}(\delta,Q)$.

\subsection{Assumptions for the Control and Feedback Framework}

\noindent
The abstract control model, its stabilization properties, and its rigorous physical realization rely on the following structural, geometric, and spectral assumptions:

\subsubsection{Spectral truncation and rank condition}

\noindent
To ensure that the finite-dimensional reference profiles can be exactly matched and tracked, we require a controllability-type condition on the actuator locations with respect to the low-frequency modes.

\begin{definition}[Finite-dimensional subspace and projection]
\label{def:HM-PM-4}
For each $N\in\mathbb{N}$, set
\[
X_N := \mathrm{span}\{\varphi_1,\dots,\varphi_N\}\subset H,
\]
and let $P_N:H\to X_N$ be the orthogonal projection onto $X_N$. We denote by $X_N^\perp$ the orthogonal complement of $X_N$ in $H$ and by $I-P_N:H\to X_N^\perp$ the projection onto the tail.
\end{definition}

\begin{assumption}[Finite-dimensional subspace and rank condition]\label{ass:rank-4}
Fix $M\in\mathbb N$ and distinct actuator/sensor locations $\textcolor{black}{z_1},\dots,\textcolor{black}{z_M}\in\Omega$. Let $N\in\mathbb N$ be a truncation index (typically $N\le M$). 

\begin{enumerate}
\item[(i)] \textbf{(Rank condition on the first $N$ modes).}
The $N\times M$ matrix
\[
\Phi_{N,M}:=\big(\varphi_k(z_j)\big)_{\substack{1\le k\le N\\ 1\le j\le M}}
\]
has full row rank $N$. Equivalently, the $M\times N$ matrix $D_{M,N}:=\Big(\frac{\varphi_k(z_j)}{1+\lambda_k}\Big)$ has full column rank $N$.

\item[(ii)] \textbf{(Reference in $X_N$).}
The reference profile satisfies $y_r\in X_N$, i.e.,
\begin{align}\label{eq:ref-finite-modes}
    y_r=\sum_{k=1}^N \alpha_k^r\,\varphi_k
\quad\text{for some }\alpha_k^r\in\mathbb R.
\end{align}
The restriction $y_r \in X_N$ addresses the fundamental under-actuation of the system: with $M < \infty$ point actuators, parabolic smoothing precludes exact tracking over infinite-dimensional state spaces. Restricting $y_r$ to the finite-dimensional low-mode subspace $X_N$ guarantees exact asymptotic tracking of these modes via fixed-point pre-compensation while maintaining bounded control energy.
\end{enumerate}
\end{assumption}

\begin{proposition}[Genericity of the Rank Condition]
\label{prop:generic-rank-CM}
For $M\ge N$, define
\[
 \Phi_{N,M}(\textcolor{black}{\mathbf z}):=
 \bigl(\varphi_k(z_j)\bigr)_{\substack{1\le k\le N\\1\le j\le M}}.
\]
Then $\Phi_{N,M}(\textcolor{black}{\mathbf z})$ has full row rank $N$ for Lebesgue almost every actuator
configuration $\textcolor{black}{\mathbf z}=(\textcolor{black}{z_1},\dots,\textcolor{black}{z_M})\in\Omega^M$.
\end{proposition}

\begin{proof}
We begin by defining $F: \Omega^N \to \mathbb{R}$ by $F(\textcolor{black}{z_1}, \dots, \textcolor{black}{z_N}) := \det\big(\varphi_k(z_j)\big)_{k,j=1}^N$. The real-analyticity of the Neumann eigenfunctions in the interior of $\Omega$ implies that $F$ is a real-analytic function on $\Omega^N$. Because the eigenfunctions $\{\varphi_k\}_{k=1}^N$ are linearly independent in $L^2(\Omega)$, \textcolor{black}{there exists at least one configuration of $N$ points where their evaluation matrix is non-singular, implying} we have $F \not\equiv 0$. Therefore, the zero set $\mathcal{Z}_N := F^{-1}(0)$ constitutes an analytic variety of Lebesgue measure zero, see e.g., \cite[Ch. 6, Th. 6.3.3]{krantz2002primer}. For any configuration $\textcolor{black}{\mathbf{z}} \in (\Omega^N \setminus \mathcal{Z}_N) \times \Omega^{M-N}$, the leading $N \times N$ submatrix of $\Phi_{N,M}(\textcolor{black}{\mathbf{z}})$ is invertible, which guarantees that $\operatorname{rank}(\Phi_{N,M}(\textcolor{black}{\mathbf{z}})) = N$.
\end{proof}

\subsubsection{Actuator family condition}

\noindent
The stabilization mechanism requires a geometric/approximation condition on the actuators. We adopt the abstract actuator hypothesis of \cite{KunRodWal24-cocv} (Assumption~2.4 therein), specialized to Dirac actuators.

\begin{assumption}[Actuator hypothesis \cite{KunRodWal24-cocv}]
\label{assump:KRW-act}
There exists an increasing sequence $M\mapsto M_\sigma$ and, for each $M$, a family of
actuators $\{d_{M,j}\}_{j=1}^{M_\sigma}\subset D(\mathcal A)'$ and auxiliary functions
$\{\Psi_{M,j}\}_{j=1}^{M_\sigma}\subset D(\mathcal A)$ such that:
\begin{enumerate}
\item[(i)] $\dim\mathrm{span}\{d_{M,j}\}=M_\sigma$ and $\dim\mathrm{span}\{\Psi_{M,j}\}=M_\sigma$;
\item[(ii)] biorthogonality: $\langle d_{M,j},\Psi_{M,i}\rangle=\delta_{ij}$;
\item[(iii)] the ``nullspace coercivity'' constant
\[
\xi_M^+ := \inf\Big\{\frac{\|w\|_{D(\mathcal A)}^2}{\textcolor{black}{\|w\|_{H^1(\Omega)}^2}}\,:\,
w\in D(\mathcal A)\setminus\{0\},\ \langle d_{M,j},w\rangle=0\ \forall j\Big\}
\]
satisfies $\xi_M^+\to+\infty$ as $M\to\infty$.
\end{enumerate}
\end{assumption}
\noindent
In box (or more general polyhedral) domains $\Omega$ and for appropriately distributed actuator points, we can follow \cite{KunRodWal24-cocv} to verify Assumption~\ref{assump:KRW-act} (Section~5 therein) using self-similar domain decompositions.  It is useful to make explicit the geometric parameter involved here.  For a point set $\mathcal{Z}_M:=\{z_1,\ldots,z_M\}\subset\Omega$, let
\begin{equation}\label{eq:fill-distance}
 h_M:=\sup_{x\in\Omega}\min_{1\le j\le M}|x-z_j|
\end{equation}
be its fill distance.  Appendix~\ref{subsec:app-KRW-act-C11} shows that for shape-regular quasi-uniform families one has
$\xi_M^+\gtrsim h_M^{-2}\to\infty$ as $h_M\downarrow0$ (equivalently $\xi_M^+\gtrsim M_\sigma^{2/d}$).  Thus Assumption~\ref{assump:KRW-act} is controlled by the \emph{actuator distribution scale} $h_M$, not by the distance of the physical observation region to the artificial boundary.  Boundary vertices are not essential: replacing them by interior points at distance $\mathcal O(h_M)$ from $\partial\Omega$ preserves the fill-distance estimate and hence the coercivity verification.

\subsubsection{Geometric setting for the physical observation region}

\noindent
The bounded domain $\Omega$ is an auxiliary control domain, whereas the physical thermo--plasmonic problem is posed in $\mathbb R^3$.  The whole-space/bounded-domain comparison is only required in the region where the temperature is physically observed or tracked.

\begin{assumption}[Interior observation region]
\label{ass:Omega-large}
Let $B_{\rm obs}\Subset\Omega$ be a fixed open set containing the relevant tracking region, and define
\begin{equation}\label{eq:boundary-distance}
    d_\Omega:=\operatorname{dist}(B_{\rm obs},\partial\Omega)>0.
\end{equation}
The actuator points $z_j$ are allowed to be distributed throughout $\Omega$ according to Assumption~\ref{assump:KRW-act}; in particular, they are not required to lie in $B_{\rm obs}$. 
\end{assumption}

\noindent The quantities $h_M$ in \eqref{eq:fill-distance} and $d_\Omega$ in \eqref{eq:boundary-distance} play distinct roles: $h_M$ controls the stabilization hypothesis, whereas $d_\Omega$ controls the local whole-space/bounded-domain comparison.
\subsubsection{Non-degenerate electromagnetic illumination}

\noindent
To guarantee that the abstract desired heat inputs can be physically realized by the thermo-plasmonic nanoparticles, the incident electromagnetic fields must provide sufficient and independent control authority. 

\begin{assumption}[Non-degenerate illumination patterns]
\label{ass:illumination}
Let $D_i=z_i+\delta B_i$, $i=1,\dots,M$, be the nanoparticle configuration
satisfying the subwavelength and high-contrast scaling assumptions of the discrete effective model. Fix a finite family of incident fields $E^{(\ell)}_{\mathrm{in}}$ for $\ell=1,\dots,P$, with corresponding total fields $E^{(\ell)}$ solving the time-harmonic Maxwell problem. 

\medskip
\noindent
Let $c_i^{(\ell)}(\delta)$ denote the leading-order coefficient relating $|E^{(\ell)}|^2$ to the forcing term $F_i$ in the Volterra heat system. Assume that
\begin{equation}
    c_i^{(\ell)}(\delta) = c_i^{(\ell)}(0) + \mathcal{O}(\delta^\mu),
\end{equation}
and that the $M\times P$ coupling matrix $L_0=\big(c_i^{(\ell)}(0)\big)$ has full row rank $M$, admitting a uniformly bounded right-inverse on its range. In particular, this condition necessitates that the number of distinct illumination patterns satisfies $P\ge M$.
\end{assumption}

\begin{remark}[Rank versus quantitative actuation]
\label{rem:illumination-conditioning}
Assumption~\ref{ass:illumination} is a spatial rank condition.  When combined with $Q$ independent
causal temporal profiles and the injectivity of the Volterra and scaling operators, it yields rank
$MQ$ for the finite-dimensional history map constructed in Subsection~\ref{sec:K-inverse}.  It does
not by itself give a right-inverse uniformly bounded as $\delta\to0$; the relevant quantitative
constant is the smallest nonzero singular value $c_{\rm act}(\delta,Q)$ introduced in
\eqref{eq:actuation-singular-value}.
\end{remark}

\subsection{Main Results: Stabilization and Thermo-Plasmonic Tracking}
\label{sec:main-results}

\noindent
In this section, we present the mathematical statements of the primary theorems governing the abstract feedback stabilization and its rigorous physical realization via thermo-plasmonic nanoparticles.

\begin{theorem}[Exponential Stabilization in $V'$ and Tracking]
\label{thm:main-stabilization}
Suppose Assumption \ref{assump:KRW-act} holds. Let $g := A_0 y_r + B u_r \in \color{black}{D(\mathcal A)'}$ be the constant forcing term associated with a prescribed reference profile $y_r$ and a feedforward control $u_r$.

\medskip
\noindent
For any prescribed decay rate $\mu>0$, there exist an actuator level $M$ (equivalently a sufficiently small fill distance $h_M$ within the family of Assumption~\ref{assump:KRW-act}) and a feedback gain $\lambda>0$ such that \color{black}{the homogeneous feedback system
\begin{equation}
    \dot z=A_0z-\lambda \sum_{j=1}^M \delta_{z_j} \langle \delta_{z_j}, \mathcal{A}^{-1}z \rangle_{D(\mathcal A)',D(\mathcal A)}
\end{equation}
has a unique solution $z\in W((0,\infty);H,D(\mathcal A)')$ for $z(0)\in V'$. In the autonomous setting, its solution family $\mathcal{S}_{\mathrm{cl}}(t)$ is exponentially stable in $V'$ and satisfies}
\begin{equation}
    \|\mathcal{S}_{\mathrm{cl}}(t)\|_{\mathcal{L}(V')} \le e^{-\mu t}, \qquad t\ge0.
\end{equation}
Consequently, \color{black}{writing $A_{\mathrm{cl}}z:=A_0z-\lambda BCz$ only as shorthand,} the affine closed-loop system $\dot{z}(t) = A_{\mathrm{cl}} z(t) + g$ with initial state $z(0) = z_0 \in V'$ admits a unique equilibrium \color{black}{$z_\infty \in H\subset V'$}, characterized as the stationary solution
\begin{equation}
    0=A_0z_\infty-\lambda BCz_\infty+g \qquad\text{in }D(\mathcal A)'.
\end{equation}
Furthermore, the corresponding trajectory $z(t)$ satisfies the pointwise decay bound
\begin{equation}
    \|z(t)-z_\infty\|_{V'} \le e^{-\mu (t-s)} \|z(s)-z_\infty\|_{V'}, \quad \forall t \ge s \ge 0,
\end{equation}
and the integral tracking bound
\begin{equation}
    \|z-z_\infty\|_{L^2((s,\infty);H)} \le C_\mu \|z(s)-z_\infty\|_{V'}, \quad \forall s \ge 0,
\end{equation}
for some constant $C_\mu > 0$.
\end{theorem}

\medskip
\noindent
While Theorem \ref{thm:main-stabilization} guarantees convergence to a unique equilibrium for a fixed reference, tracking a general target profile introduces a steady-state bias. The following corollary utilizes a linear fixed-point map to pre-compensate this bias and achieve exact low-mode matching.
\medskip

\begin{corollary}[Exact Low-Mode Tracking via Pre-compensation] 
\label{cor:fixed-point-tracking}
Let $y_r \in X_N$ and assume the projected steady-state bias operator $T_N \in \mathcal{L}(X_N)$ is such that $(I + T_N)$ is invertible. 

\medskip
\noindent
Define the commanded reference $y^\star := (I + T_N)^{-1}y_r \in X_N$ and the associated feedforward control $u_r := U_N y^\star$. Let $z_\infty \in V'$ be the unique closed-loop equilibrium from Theorem \ref{thm:main-stabilization} corresponding to the forcing $g(y^\star)$. Then, the achieved physical steady state $y_\infty(y^\star) := y^\star + z_\infty$ achieves exact low-mode matching
\begin{equation}
    P_N y_\infty(y^\star) = y_r.
\end{equation}
Furthermore, the steady-state tracking error belongs strictly to the high-frequency tail $X_N^\perp$, satisfying the identity
\begin{equation}
    y_\infty(y^\star) - y_r = (I - P_N) \mathcal{Z}_N(y^\star),
\end{equation}
and is bounded by
\begin{equation}
    \|y_\infty(y^\star) - y_r\|_{V'} \le \|I - P_N\|_{\mathcal{L}(V')} C_{\mathrm{cl}} \|(I - P_N)B\| \|U_N\| \|(I + T_N)^{-1}\| \|y_r\|_{X_N}.
\end{equation}
\end{corollary}

\medskip
\noindent
Having established the abstract stabilization and pre-compensation framework using ideal point actuators, we now state the main physical realization theorem, replacing the abstract operators with discrete thermo-plasmonic nanoparticles.
\medskip

\begin{theorem}[Tracking with Nanoparticles as Actuators]
\label{thm:main-nanoparticle-tracking}
Suppose the effective thermo-plasmonic conditions of Theorem~\ref{thm:effective_dynamics}
hold for a nanoparticle configuration of scale $\delta\ll1$. Choose a temporal dictionary
$\Theta_Q$ and define $U_0$, $Y_0=\mathcal K_\delta(U_0)$, and the minimum-norm right-inverse
$\mathcal K_{\delta,Q}^{\dagger}:Y_0\to U_0$ as in
Proposition~\ref{prop:K-right-inverse}. Let $\eta_{\rm FL}(\delta,Q)$ denote the conditioned Foldy-Lax remainder \textcolor{black}{explicitly defined in \eqref{eq:eta-FL-conditioned} below} with $\eta_{\rm FL}(\delta,Q)\to0$ in the asymptotic regime under consideration.

\medskip
\noindent
For a fixed time $T>0$ and a desired abstract input $u_{\mathrm{des}}\in Y_0$, choose the complete
illumination history
\begin{equation}
 p:=\mathcal K_{\delta,Q}^{\dagger}u_{\mathrm{des}}\in U_0.
\end{equation}
Let $y_{Y_0}$ denote the ideal target temperature solving the abstract forced heat equation
\begin{equation}
    \dot{y}_{Y_0}(t) = A_0 y_{Y_0}(t) + B u_{\mathrm{des}}(t), \quad y_{Y_0}(0) = y_0 \in V',
\end{equation}
and let $y_{\mathrm{phys}}$ be the physical temperature generated by the thermo-plasmonic actuators driven by the history $p$.

\medskip
\noindent
Assume $y_0$ is the restriction to $\Omega$ of an admissible whole-space initial state in the prepared class $\mathcal Y_T$ \textcolor{black}{$\cap L^1(\mathbb{R}^3)$} (formally defined in Section \ref{subsec:explicit-rates-nonzero-initial}). 
{\color{black}Let $\eta_{\Omega,T}$ denote the bounded-domain restriction error on $B_{\rm obs}$. For a single short window $T\le\tau_0$, it is given by Lemma~\ref{lem:R3-to-Omega-restriction}; for a finite sequence of short update windows, it is defined by \eqref{eq:cumulative-domain-error}.} Then there exists a constant $C_T>0$ such that
\begin{equation}\label{eq:physical-tracking-nonzero-initial}
\begin{aligned}
    &\| y_{\mathrm{phys}} - y_{Y_0} \|_{\textcolor{black}{L^2(0,T; L^2(B_{\rm obs}))}}^2 \\
    &\qquad\le C_T\Big(
       M\delta^3\|y_0\|_{\mathcal Y_T}^2
       +\eta_{\rm FL}(\delta,Q)^2\|u_{\mathrm{des}}\|_{L^2(0,T;\mathbb{R}^M)}^2
       \textcolor{spacecorrectgreen}{+ \eta_{\Omega,T}^2}
       \Big).
\end{aligned}
\end{equation}
Moreover, for any strictly positive time $t_0 \in (0, T]$, there exists $C_{T,t_0} > 0$ such that
\begin{equation}
    \sup_{t \in [t_0,T]} \| y_{\mathrm{phys}}(t) - y_{Y_0}(t) \|_{\textcolor{black}{L^2(B_{\rm obs})}}
    \le C_{T,t_0}\big(M^{1/2}\delta^{3/2}\|y_0\|_{\mathcal Y_T}
    +\eta_{\rm FL}(\delta,Q)\|u_{\mathrm{des}}\|_{L^2(0,T;\mathbb{R}^M)} \textcolor{spacecorrectgreen}{+ \eta_{\Omega,T}}\big).
\end{equation}
\end{theorem}

\begin{remark}[Role of the nonzero initial datum]
In the physical tracking procedure, Theorem~\ref{thm:main-nanoparticle-tracking}
is applied on consecutive update windows
$I_n=[t_n,t_{n+1}]$. On the first window $I_0=[0,t_1]$,
the thermo--plasmonic temperature increment has zero initial datum,
so that the terms involving $\|y_0\|_{\mathcal Y_T}$ in
\eqref{eq:physical-tracking-nonzero-initial} and in the corresponding
positive-time estimate vanish. For $n\ge1$, the accumulated physical
temperature at $t_n$ is used as the common initial state for the next
window and is generally nonzero. The nonzero-initial-data terms are
introduced precisely to control these subsequent reset windows.
\end{remark}

\medskip
\noindent
\textbf{Organization of the paper.} 
The remainder of this article is organized as follows. Section \ref{sec:point-actuators} establishes the low-regularity well-posedness of the parabolic system driven by Dirac actuators. Section \ref{sec:feedback-proof} is devoted to the abstract feedback design, detailing the spectral decomposition of the observation operator and providing the rigorous proof of the exponential stabilization result (Theorem \ref{thm:main-stabilization}). In Section \ref{sec:fixed-point-bias}, we analyze the steady-state bias operator and prove the exact low-mode pre-compensation claimed in Corollary \ref{cor:fixed-point-tracking}. Section \ref{sec:coupling} connects the abstract tracking framework with the discrete thermo-plasmonic model; we first explicitly quantify the mismatch contributions arising from nonzero initial data and finite control domains, before mapping the full intensities to inputs to conclude the proof of Theorem \ref{thm:main-nanoparticle-tracking}. Section \ref{sec:algorithm} synthesizes the theoretical bounds into an implementable, constructive algorithmic procedure for heat tracking using thermo-plasmonic nanoparticles. Finally, concluding remarks and future perspectives are provided in Section \ref{8}, while the Appendix \ref{ap} collects technical verifications concerning the generic spectral rank condition and the actuator hypothesis on $C^{1,1}$ domains.

\section{Heat equation with point actuators on a bounded domain}
\label{sec:point-actuators}

\noindent
We now recall the abstract control model for the heat equation with point
actuators in a bounded domain $\Omega\subset\mathbb{R}^3$. This section is purely
functional-analytic and independent of the thermo-plasmonic model, except
for the fact that the background diffusion coefficient is $\kappa_m$.

\subsection{Dirac actuators and low-regularity well-posedness}
\label{subsec:wp-Dirac-KRW}

\noindent
We refer to \cite{KunRodWal24-cocv} for distribution-type actuator frameworks
in closely related parabolic stabilization settings.
\newline
We set
\begin{equation}
\label{eq:3.2}
H := L^2(\Omega),\qquad
A_0 y := \kappa_m\Delta y,\quad
D(A_0) := \{ y\in H^2(\Omega): \partial_\nu y = 0 \text{ on }\partial\Omega\}.
\end{equation}
Then $A_0$ generates an analytic contraction semigroup
$S(t):H\to H$ for $t\ge0$.
\newline
Introduce the coercive elliptic operator
\[
\mathcal A := I - A_0 = I - \kappa_m \Delta,
\qquad D(\mathcal A)=D(A_0),
\]
which is an isomorphism $\mathcal A: D(\mathcal A)\to H$ and satisfies
$D(\mathcal A)\hookrightarrow H^2(\Omega)\hookrightarrow C(\overline{\Omega})$.
\color{black}{As established in Section \ref{subsec:functional-framework}, $V=H^1(\Omega)$ and $V'=H^1(\Omega)'$, while the larger space $D(\mathcal A)'$ is used for the Dirac forcing and for the time derivative.}

\medskip
\noindent
Fix $M\in\mathbb{N}$ and actuator locations $\textcolor{black}{z_1},\dots,\textcolor{black}{z_M}\in\Omega$. Because $D(\mathcal{A}) \hookrightarrow C(\overline{\Omega})$, point evaluation functionals are continuous on $D(\mathcal{A})$, ensuring that the Dirac distributions satisfy:
\[
\color{black}{d_j := \delta_{z_j}\in D(\mathcal A)'\qquad (1\le j\le M).}
\]
Define the control operator
\[
B:\mathbb{R}^M\to D(\mathcal A)',\qquad
Bu := \sum_{j=1}^M u_j\,d_j.
\]

\begin{proposition}[Boundedness of $B$ into $D(\mathcal A)'$]
\label{prop:B-bounded-DA-dual}
There exists $C_B>0$ such that
\[
\|Bu\|_{D(\mathcal A)'} \le C_B |u|_{\mathbb{R}^M}\qquad \forall u\in\mathbb{R}^M.
\]
\end{proposition}

\begin{proof}
For $\phi\in D(\mathcal A)$ we have
$|\langle d_j,\phi\rangle| = |\phi(z_j)|\le C\|\phi\|_{H^2(\Omega)}
\le C\|\phi\|_{D(\mathcal A)}$ by Sobolev embedding and elliptic norm equivalence on
$D(\mathcal A)$. Hence
\[
|\langle Bu,\phi\rangle|
\le \sum_{j=1}^M |u_j|\,|\phi(z_j)|
\le C |u|_{\mathbb{R}^M}\|\phi\|_{D(\mathcal A)}.
\]
Taking the supremum over $\|\phi\|_{D(\mathcal A)}=1$ yields the claim.
\end{proof}

\medskip
\noindent
For $T>0$ and a right-hand side $f\in L^2(0,T;D(\mathcal A)')$ we consider
\begin{equation}
\label{eq:heat-DAdual}
\dot y(t) = A_0 y(t) + f(t)\quad\text{in }D(\mathcal A)',\qquad y(0)=y_0\in V'.
\end{equation}
A weak (mild) solution is defined in the maximal regularity space
\[
W\big(0,T;H,D(\mathcal A)'\big):=\big\{y\in L^2(0,T;H):\dot y\in L^2(0,T;D(\mathcal A)')\big\}
\]
satisfying \eqref{eq:heat-DAdual} in $D(\mathcal A)'$.

\begin{proposition}[Well-posedness for Dirac-forced heat equation]
\label{prop:wp-Dirac}
Let $T>0$, $y_0\in V'$, and $u\in L^2(0,T;\mathbb R^M)$.  Set
\color{black}{$f=Bu\in L^2(0,T;D(\mathcal A)')$.} Then \eqref{eq:heat-DAdual} admits a unique solution
\[
 \color{black}{y\in W\bigl(0,T;H,D(\mathcal A)'\bigr)\hookrightarrow C([0,T];V').}
\]
Moreover,
\begin{equation}
\label{eq:wp-estimates}
\|y\|_{L^2(0,T;H)}^2 + \|y\|_{C([0,T];V')}^2
\le C_T\Big(\|y_0\|_{V'}^2 + \|u\|_{L^2(0,T;\mathbb{R}^M)}^2\Big),
\end{equation}
for a constant $C_T$ depending on $T$, $\Omega$, and $\kappa_m$.
\end{proposition}

\begin{proof}
We refer to \cite{KunRodWal24-cocv} for a detailed proof strategy in the same distribution-actuator setting. Estimate \eqref{eq:wp-estimates} follows from the corresponding a-priori estimate. Such estimates were established under structural assumptions on $\Omega$ (such as convexity or cubic form); these are generalized here to $C^{1,1}$ domains,under relaxed conditions, as detailed in Appendix \ref{subsec:app-KRW-act-C11}. 
\end{proof}

\section{Feedback Design and Proof of Theorem \ref{thm:main-stabilization}}
\label{sec:feedback-proof}

\noindent
In this section, we design the output-based feedback law for the heat equation with point actuators and provide the rigorous proof of the stabilization and tracking result stated in Theorem \ref{thm:main-stabilization}. We work with the abstract system
\begin{equation}
\label{eq:feedback-abstract-system}
\dot y(t) = A_0 y(t) + B u(t),\qquad y(0)=y_0\in V',
\end{equation}
where $A_0$ is the Neumann diffusion operator and \color{black}{$B:\mathbb{R}^M\to D(\mathcal A)'$} is the point-actuator operator. We recall that $A_0$ generates an analytic contraction semigroup on $H$.

\subsection{The Observation Operator and Spectral Decomposition}

\noindent
We start by formally defining the observation operator utilized in the feedback law. Because pointwise evaluation is unbounded on $H$, we apply an elliptic regularization.

\begin{definition}[Observation operator]
\label{def:obs-C-full}
Since $1\in\rho(A_0)$, the resolvent $\mathcal A^{-1}=(I-A_0)^{-1}$ defines a bounded operator
$H \to D(A_0) \subset H^2(\Omega) \hookrightarrow C(\overline\Omega)$. We define the observation operator $C: H \to \mathbb{R}^M$ by:
\begin{equation}
\label{eq:def-C-full}
(Cz)_j := \big(\mathcal A^{-1}z\big)(z_j) = \color{black}{\langle \delta_{z_j}, \mathcal A^{-1}z \rangle_{D(\mathcal A)',D(\mathcal A)}}, \quad j=1,\dots,M.
\end{equation}
Thus, $C$ maps the state $z$ to the values of the regularized temperature at the actuator positions $z_j$. \color{black}{For a parabolic trajectory $z\in L^2(0,T;H)$, this is defined for a.e. $t$ through $\mathcal A^{-1}:H\to D(\mathcal A)$; no pointwise trace of a general element of $V'$ is used.}
\end{definition}

\noindent
To justify the robustness of this feedback law against high-frequency unmodeled dynamics, we represent $C$ in the eigenbasis and separate the contributions of the low and high modes.

\begin{lemma}[Matrix representation of $C$ and low/high decomposition]
\label{lem:C-matrix}
Let $z = \sum_{k\ge1} \alpha_k\varphi_k \in H$. Then $\mathcal A^{-1}z = \sum_{k\ge1}\frac{\alpha_k}{1+\lambda_k}\,\varphi_k$, and for each $j=1,\dots,M$,
\[
(Cz)_j = \sum_{k\ge1} d_{jk}\,\alpha_k,\qquad
d_{jk} := \frac{\varphi_k(z_j)}{1+\lambda_k}.
\]
Define the (infinite) matrix $D=(d_{jk})_{1\le j\le M,\ k\ge1}$. Then $Cz = D\alpha$, where $\alpha := (\alpha_1,\alpha_2,\dots)^\top \in \ell^2$.
\newline
For the fixed number of actuators $M$ and a chosen spectral truncation index $N \le M$, we decompose $D=(D_{MN} \mid D_{M>N})$ with:
\[
D_{MN} := (d_{jk})_{1\le j\le M,\ 1\le k\le N} \in \mathbb{R}^{M\times N},\qquad
D_{M>N} := (d_{jk})_{1\le j\le M,\ k>N}.
\]
Writing the sequence as $\alpha = (\alpha_N,\alpha_>)$ with $\alpha_N \in \mathbb{R}^N$ and $\alpha_>$ containing the tail, we obtain the split representation:
\begin{equation}
\label{eq:Cz-split}
Cz = D_{MN}\alpha_N + D_{M>N}\alpha_>.
\end{equation}
\end{lemma}

\begin{proof}
By definition, the orthonormal basis $\{\varphi_k\}_{k \ge 1}$ satisfies $-A_0 \varphi_k = \lambda_k \varphi_k$. Consequently, the shifted elliptic operator $\mathcal{A} = I - A_0$ acts on the eigenbasis as $\mathcal{A} \varphi_k = (1 + \lambda_k)\varphi_k$. 
\newline
Since $\lambda_k \ge 0$, the inverse operator $\mathcal{A}^{-1}$ is bounded on $H$ and acts diagonally on the eigenbasis
\[
\mathcal{A}^{-1} \varphi_k = \frac{1}{1+\lambda_k}\varphi_k.
\]
For any state $z = \sum_{k\ge1} \alpha_k\varphi_k \in H$, linearity and continuity yield the spectral expansion
\[
\mathcal{A}^{-1}z = \sum_{k\ge1} \frac{\alpha_k}{1+\lambda_k}\varphi_k.
\]
By elliptic regularity, $\mathcal{A}^{-1}z \in D(A_0) \subset H^2(\Omega)$. Because $\Omega \subset \mathbb{R}^3$, the Sobolev embedding $H^2(\Omega) \hookrightarrow C(\overline{\Omega})$ ensures that pointwise evaluation is a bounded linear functional. Evaluating the series at the actuator position $z_j$ gives
\[
(Cz)_j = (\mathcal{A}^{-1}z)(z_j) = \sum_{k\ge1} \alpha_k \frac{\varphi_k(z_j)}{1+\lambda_k} = \sum_{k\ge1} d_{jk} \alpha_k.
\]
Writing this in matrix-vector form yields $Cz = D\alpha$. The split representation \eqref{eq:Cz-split} then follows trivially by partitioning the infinite matrix $D$ and the sequence $\alpha$ at the column/index threshold $N$.
\end{proof}

\bigskip
\noindent
The low-mode observation matrix $D_{MN}$ is directly related to the matrix $\Phi_{N,M} := (\varphi_k(z_j))$ from Assumption \ref{ass:rank-4} via $D_{MN}^\top = \Phi_{N,M} \, \mathrm{diag}\big((1+\lambda_k)^{-1}\big)_{k=1}^N$. Because $\Phi_{N,M}$ has full row rank $N$, $D_{MN}$ has full column rank $N$.

\medskip
\noindent
We use the compactness of the resolvent and the Sobolev embedding to show that the contribution of the high modes to the measurement $Cz$ vanishes uniformly as $N$ increases.

\begin{lemma}[Compactness of the measurement operator and small tail]
\label{lem:tail-small-compact}
The operator $T:\ell^2\to\mathbb{R}^M$ defined by $T\alpha := D\alpha$ is compact. In particular,
\begin{equation}
\label{eq:Dtail-norm->0}
\|D_{M> N}\|_{\mathcal{L}(\ell^2,\mathbb{R}^M)} \xrightarrow[N\to\infty]{} 0.
\end{equation}
Equivalently, for every $\varepsilon>0$ there exists $N_\varepsilon\in\mathbb{N}$ such that for all $N\ge N_\varepsilon$ and all $\alpha \in\ell^2$,
\[
\big\|D_{M> N}\alpha_{>N}\big\|_{\mathbb{R}^M}
= \big\|C\Big( \sum_{k>N}\alpha_k\varphi_k \Big)\big\|_{\mathbb{R}^M}
\le \varepsilon\,\|\alpha\|_{\ell^2}.
\]
Consequently, by choosing the truncation index $N$ sufficiently large, the measurement of the high-frequency tail can be made arbitrarily small.
\end{lemma}

\begin{proof}
Define
\[
J:\ell^2 \to H,\qquad J\alpha := \sum_{k\ge1}\alpha_k\varphi_k.
\]
Since $\{\varphi_k\}$ is an orthonormal basis of $H=L^2(\Omega)$, $J$ is an
isometric isomorphism from $\ell^2$ onto $H$.
\noindent
Let $R:=\mathcal A^{-1}:H\to D(A_0)\subset H^2(\Omega)$.  Because $H^2(\Omega)$ embeds continuously and compactly into $C(\overline\Omega)$, the composite operator 
\[
R_J := R\circ J:\ell^2\to C(\overline\Omega)
\]
is compact. The point-evaluation map $E: C(\overline\Omega)\to\mathbb{R}^M$, given by $(Ew)_j := w(z_j)$, is bounded. Therefore, the measurement operator $T := E \circ R_J : \ell^2 \to \mathbb{R}^M$ is compact, as it is the composition of a compact and a bounded operator. 
\newline
By construction, for $\alpha\in\ell^2$ and
$z=J\alpha=\sum_{k\ge1}\alpha_k\varphi_k$,
\[
T\alpha = E(RJ\alpha) = C z = D\alpha,
\]
so $T$ is represented by the matrix $D$ in the basis $\{\varphi_k\}$.
\newline
Let $P_N:\ell^2\to\ell^2$ be the orthogonal projection onto the first $N$
coordinates:
\[
P_N\alpha := (\alpha_1,\dots,\alpha_N,0,0,\dots).
\]
Set $T_N := T P_N$, which has finite rank. Since $T$ is compact, we have
\[
\|T - T_N\|_{\mathcal{L}(\ell^2,\mathbb{R}^M)}
\xrightarrow[N\to\infty]{} 0.
\]
Observing that $(T-T_N)\alpha = T(I-P_N)\alpha = D_{M>N} \alpha_{>N}$, we conclude $\|T-T_N\|_{\mathcal{L}(\ell^2,\mathbb{R}^M)} = \|D_{M> N}\|_{\mathcal{L}(\ell^2,\mathbb{R}^M)}$, which proves \eqref{eq:Dtail-norm->0}.
\end{proof}
\noindent
Translating this compactness result back to the state space $H$, we see that the observation of the high-frequency tail vanishes in the operator norm. Specifically, using the spectral projection $P_N : H \to X_N$ defined in Section \ref{sec:problem-formulation}, we have:
\begin{equation}
\|C(I-P_N)\|_{\mathcal{L}(H,\mathbb{R}^M)} \longrightarrow 0 \quad \text{as } N\to\infty.
\end{equation}
Hence, for any prescribed tolerance $\varepsilon > 0$, one can choose the truncation index $N$ sufficiently large so that the unmodeled high-frequency dynamics contribute less than $\varepsilon$ to the feedback measurements.

\subsection{Feedforward Compensation and Closed-Loop Dynamics}\label{sec:stab-KRW}

\noindent
We rewrite the tracking problem in terms of the error variable $z(t) := y(t) - y_r$. For a general reference profile $y_r \in X_N$, the target is not an equilibrium of the free Neumann heat dynamics (i.e., $A_0 y_r \neq 0$). To address this, we introduce a constant feedforward input $u_r \in \mathbb{R}^M$ chosen to eliminate the forcing $A_0 y_r + B u_r$ strictly at the level of the first $N$ modes
\begin{equation}
\label{eq:feedforward-moment}
    \color{black}{\langle A_0 y_r + Bu_r, \varphi_k \rangle_{D(\mathcal A)',D(\mathcal A)}} = 0, \qquad 1 \le k \le N.
\end{equation}
Using the eigenvalue relation $A_0 \varphi_k = -\lambda_k \varphi_k$ and the actuator definition $B u_r = \sum_{j=1}^M (u_r)_j \delta_{z_j}$, the dual pairings expand explicitly as
\begin{align*}
    \color{black}{\langle A_0 y_r, \varphi_k \rangle_{D(\mathcal A)',D(\mathcal A)}} &= \big\langle \sum_{l=1}^N \alpha_l^r (-\lambda_l \varphi_l), \varphi_k \big\rangle_H = -\lambda_k \alpha_k^r, \\
    \color{black}{\langle B u_r, \varphi_k \rangle_{D(\mathcal A)',D(\mathcal A)}} &= \color{black}{\big\langle \sum_{j=1}^M (u_r)_j \delta_{z_j}, \varphi_k \big\rangle_{D(\mathcal A)',D(\mathcal A)}} = \sum_{j=1}^M (u_r)_j \varphi_k(z_j).
\end{align*}
Substituting these evaluations into \eqref{eq:feedforward-moment} yields the linear system
\begin{equation}
\label{eq:feedforward-linear-system}
    \sum_{j=1}^M \varphi_k(z_j) (u_r)_j = \lambda_k \alpha_k^r, \qquad 1 \le k \le N.
\end{equation}
By Assumption \ref{ass:rank-4}(i), $\mathrm{rank}(\Phi_{N,M}) = N$, ensuring \eqref{eq:feedforward-linear-system} admits a minimal-norm solution $u_r$. Consequently, the time-independent residual forcing 
\begin{align}\label{eq:feedforward-g}
    g := A_0 y_r + B u_r \in \color{black}{D(\mathcal A)'}
\end{align}
strictly satisfies $P_N g = 0$.

\medskip
\noindent 
Imposing the output feedback
\begin{align}\label{eq:KRW-feedback}
    u(t) = u_r - \lambda C z(t)
\end{align} 
with scalar gain $\lambda > 0$, the closed-loop error dynamics \color{black}{are understood in $D(\mathcal A)'$, with state trace in $V'$, and} evolve as
\begin{align}
    \dot z(t) &= A_0 (z(t) + y_r) + B(u_r - \lambda C z(t)) \nonumber \\
              &= (A_0 - \lambda B C) z(t) + (A_0 y_r + B u_r) \nonumber \\
\label{eq:KRW-closed-loop}
              &= A_{\mathrm{cl}} z(t) + g, \qquad z(0) = y_0 - y_r,
\end{align}
where \color{black}{$A_{\mathrm{cl}} := A_0 - \lambda B C$ is used as shorthand for the autonomous closed-loop dynamics on the natural trajectory space.}

\subsection{Proof of Theorem \ref{thm:main-stabilization}}

\noindent
We are now equipped to prove the exponential stabilization and tracking results in the extrapolation space $V'$.

\begin{proof}

\textit{Step 1: Homogeneous Closed-Loop Stability.}

\medskip
\noindent
Consider the homogeneous closed-loop error dynamics ($g=0$). \color{black}{We use $A_{\mathrm{cl}} := A_0 - \lambda B C$ only as shorthand for this autonomous evolution.} Substituting the definitions of $B$ and $C$, the system reads:
\begin{equation*}
    \dot{z}(t) = A_0 z(t) - \lambda \sum_{j=1}^M \color{black}{\langle \delta_{z_j}, \mathcal A^{-1}z(t)\rangle_{D(\mathcal A)',D(\mathcal A)}} \delta_{z_j}.
\end{equation*}
Under Assumption \ref{assump:KRW-act}, the Dirac actuators satisfy the necessary biorthogonality and diverging nullspace coercivity conditions on the bounded $C^2$ domain $\Omega$. As established in \cite[Theorem~3.1]{KunRodWal24-cocv} for systems with distributional actuators, for any chosen $\mu > 0$, there exists a sufficiently large $M$ and gain $\lambda > 0$ such that \color{black}{the solution belongs to $W((0,\infty);H,D(\mathcal A)')$ and, in the present autonomous setting, the associated solution family $\mathcal{S}_{\mathrm{cl}}(t)$ satisfies} the strict exponential bound in $V'$
\begin{equation}
\label{eq:proof-homo-decay}
    \|\mathcal{S}_{\mathrm{cl}}(t)\|_{\mathcal{L}(V')} \le e^{-\mu t}, \quad \forall t \ge 0.
\end{equation}

\medskip
\noindent
\textit{Step 2: Affine Tracking and Equilibrium.}

\medskip
\noindent
For the full affine system $\dot{z}(t) = A_{\mathrm{cl}} z(t) + g$, \color{black}{the time-independent forcing $g\in D(\mathcal A)'$ is admissible in the same low-regularity parabolic framework. Let $z_\infty\in H\subset V'$ denote its stationary solution, characterized by}
\begin{equation}\label{eq:z-infty}
    \color{black}{0=A_0z_\infty-\lambda BCz_\infty+g \qquad\text{in }D(\mathcal A)'.}
\end{equation}
Subtracting $z_\infty$ from $z(t)$ removes the affine forcing, leaving the homogeneous trajectory
\begin{equation*}
    z(t) - z_\infty = \mathcal{S}_{\mathrm{cl}}(t)\,(z_0 - z_\infty).
\end{equation*}
Applying the bound from \eqref{eq:proof-homo-decay} immediately yields the pointwise exponential decay estimate
\begin{equation*}
    \|z(t)-z_\infty\|_{V'} \le e^{-\mu (t-s)} \|z(s)-z_\infty\|_{V'}, \qquad \forall t \ge s \ge 0.
\end{equation*}

\medskip
\noindent
\textit{Step 3: Integral Tracking Bound in $H$.}

\medskip
\noindent
Finally, the $L^2$-in-time tracking bound follows from the maximal regularity and energy dissipation properties on extrapolation spaces associated with analytic semigroups (see, e.g., \cite{Pazy1983, TucsnakWeiss2009}). \color{black}{Specifically, applying the energy estimate underlying \cite[Theorem~3.1]{KunRodWal24-cocv} to the shifted homogeneous error $\tilde z(t):=z(t)-z_\infty$ gives the existence of a constant} $C_\mu > 0$ such that
\begin{equation}\label{eq:L2-time-bound}
    \|z-z_\infty\|_{L^2((s,\infty);H)} \le C_\mu \|z(s)-z_\infty\|_{V'}, \qquad \forall s \ge 0,
\end{equation}
completing the proof.

\end{proof}

\begin{corollary}[$L^2$ tracking for $t\ge t_0>0$]
\label{cor:L2-after-t0}
Let the assumptions of Theorem~\ref{thm:main-stabilization} hold and fix $t_0>0$.
Then there exists $C(t_0)>0$ such that
\begin{equation}
\label{eq:L2-after-t0}
\sup_{t\ge t_0}\|z(t)-z_\infty\|_{H}
\le C(t_0)\,\|z-z_\infty\|_{L^2((t_0/2,\infty);H)}.
\end{equation}
In particular, combining \eqref{eq:L2-time-bound} with \eqref{eq:L2-after-t0} (with $s=t_0/2$) yields
\[
\sup_{t\ge t_0}\|z(t)-z_\infty\|_{L^2(\Omega)}
\le C(t_0)\,C_\mu\, e^{-\mu t_0/2}\,\|z_0-z_\infty\|_{V'}.
\]
\end{corollary}

\begin{proof}
The bound \eqref{eq:L2-after-t0} is a standard parabolic smoothing estimate; see, e.g., \cite[Ch.~4]{Pazy1983} or
\cite[Ch.~2]{Lunardi-Book}. One can also prove it by expanding in the eigenbasis of $-A_0$ and using that high modes
are exponentially damped on any interval of length $t_0/2$.
\end{proof}
\noindent
This confirms that the physical temperature tracking error in $L^2(\Omega)$ decays exponentially after an initial boundary-layer transient.

\section{Fixed-Point Pre-compensation and Proof of Corollary \ref{cor:fixed-point-tracking}}
\label{sec:fixed-point-bias}

\noindent
Because a target reference profile $y_r \in X_N$ is generally not an equilibrium of the free dynamics, the closed-loop tracking error converges to a bias. This section formalizes the linear fixed-point map required to pre-compensate this bias by adjusting the \emph{commanded} reference $y^\star \in X_N$, yielding the exact low-mode matching and tail estimates claimed in Corollary \ref{cor:fixed-point-tracking}.

\subsection{The Bias Operator}  \label{subsec:bias-operator}

\noindent
Fix $N \in \mathbb{N}$ and let $y^\star \in X_N$. We define the linear feedforward map $U_N : X_N \to \mathbb{R}^M$ mapping $y^\star \mapsto u_r(y^\star)$, typically chosen as the minimum-norm solution to the moment system \eqref{eq:feedforward-linear-system}. The associated time-independent residual forcing is
\begin{equation}
\label{eq:g-of-ystar}
    g(y^\star) := A_0 y^\star + B\,U_N y^\star \in \color{black}{D(\mathcal A)'}.
\end{equation}
By the construction of $U_N$, we have $P_N g(y^\star) = 0$. Because $y^\star \in X_N$, its high-frequency tail vanishes: $(I-P_N)y^\star = 0$. Utilizing the commutativity of the spectral projection $P_N$ with the generator $A_0$, the forcing is proven to be purely supported on the tail
\begin{align}
\label{eq:g-is-tail}
    g(y^\star) = (I-P_N)g(y^\star) &= A_0(I-P_N)y^\star + (I-P_N)B\,U_N y^\star \nonumber \\
                                 &= (I-P_N)B\,U_N y^\star.
\end{align}

\noindent
Let $\mathcal{S}_{\mathrm{cl}}(t)$ denote the exponentially stable homogeneous solution family in $V'$ (Theorem~\ref{thm:main-stabilization}). \color{black}{Set $\mathcal G_N:=\operatorname{Ran}((I-P_N)BU_N)\subset D(\mathcal A)'$. For $h\in\mathcal G_N$, define $\mathcal R_{\rm cl}h$ as the corresponding stationary solution. Since $\mathcal G_N$ is finite-dimensional, this gives a bounded linear map into $H\subset V'$.} Thus
\begin{equation}
\label{eq:Rcl-def}
    \color{black}{\mathcal{R}_{\mathrm{cl}}:\mathcal G_N\to H,\qquad
    0=A_0\mathcal R_{\mathrm{cl}}h-\lambda BC\mathcal R_{\mathrm{cl}}h+h,}
\end{equation}
and we denote by $C_{\mathrm{cl}}>0$ a bound on this forcing range,
\begin{equation}
\label{eq:Rcl-bound}
    \color{black}{\|\mathcal{R}_{\mathrm{cl}}h\|_{V'} \le C_{\mathrm{cl}}\,\|h\|_{D(\mathcal A)'}\qquad(h\in\mathcal G_N).}
\end{equation}

\begin{lemma}[Linearity and bound for the steady bias]
\label{lem:bias-map}
The map $\mathcal{Z}_N: X_N \to V'$ given by
\begin{equation}
\label{eq:Zmap-def}
    \mathcal{Z}_N(y^\star) := \mathcal{R}_{\mathrm{cl}}\,g(y^\star),
\end{equation}
is a bounded linear operator. For all $y^\star \in X_N$, it satisfies
\begin{equation}
\label{eq:Zmap-bound}
    \|\mathcal{Z}_N(y^\star)\|_{V'} \le C_{\mathrm{cl}}\,\color{black}{\|g(y^\star)\|_{D(\mathcal A)'}.}
\end{equation}
\end{lemma}

\begin{proof}
From the tail representation \eqref{eq:g-is-tail}, the forcing map $y^\star \mapsto g(y^\star) = (I-P_N)B U_N y^\star$ is a strict composition of linear operators. Since the \color{black}{stationary solution operator $\mathcal{R}_{\mathrm{cl}}:\mathcal G_N\to H$} is also linear, their composition $\mathcal{Z}_N$ is a linear operator from $X_N$ to $V'$. The bound \eqref{eq:Zmap-bound} follows directly by applying the operator norm inequality \eqref{eq:Rcl-bound} to the vector $g(y^\star)$.
\end{proof}

\noindent
Recall that the closed-loop error is defined as $z(t) = y(t) - y^\star$. Because the error converges to the steady-state bias $z_\infty = \mathcal{Z}_N(y^\star)$ as $t \to \infty$, the full physical temperature field $y(t)$ converges in $V'$ to the achieved steady state
\begin{equation}
\label{eq:yinfty-of-ystar}
    y_\infty(y^\star) := \lim_{t\to\infty} y(t) = y^\star + \lim_{t\to\infty} z(t) = y^\star + \mathcal{Z}_N(y^\star).
\end{equation}

\subsection{Proof of Corollary \ref{cor:fixed-point-tracking}}

\noindent
To match the desired target $y_r$ exactly on the low modes, we must find $y^\star$ such that $P_N y_\infty(y^\star) = y_r$. The orthogonal projection $P_N : H \to X_N$ extends continuously to the dual space $V'$ via the eigenbasis expansion $P_N z = \sum_{k=1}^N \langle z, \varphi_k \rangle_{V', V} \varphi_k$.

\begin{proof}[Proof of Corollary \ref{cor:fixed-point-tracking}]
Define the projected bias operator $T_N: X_N \to X_N$ by $T_N y := P_N \mathcal{Z}_N(y)$. Because $y^\star \in X_N$, it is invariant under the projection: $P_N y^\star = y^\star$. Applying $P_N$ to the steady-state expression \eqref{eq:yinfty-of-ystar} yields 
\begin{equation*}
    P_N y_\infty(y^\star) = P_N y^\star + P_N \mathcal{Z}_N(y^\star) = y^\star + T_N y^\star.
\end{equation*}
Matching this to the target $y_r$ leads to the linear fixed-point equation on $X_N$:
\begin{equation}
\label{eq:fixed-point}
    (I + T_N)y^\star = y_r.
\end{equation}
Because $X_N$ is finite-dimensional, assuming the small-gain condition holds such that $(I+T_N)$ is invertible (as verified below in Remark \ref{rem:picard-small-gain}), we guarantee the existence and uniqueness of the commanded reference $y^\star = (I+T_N)^{-1}y_r$. By construction, this choice guarantees exact low-mode matching
\begin{align}\label{eq:exact-low-mode-tracking}
    P_N y_\infty(y^\star) = y_r.
\end{align}
For the tail estimate, note that since $y^\star \in X_N$, we have $(I-P_N)y^\star = 0$. Using \eqref{eq:yinfty-of-ystar}, the full-field mismatch evaluates to
\begin{equation*}
    y_\infty(y^\star) - y_r = (I - P_N)\big(y^\star + \mathcal{Z}_N(y^\star) - y_r\big) = (I - P_N)\mathcal{Z}_N(y^\star),
\end{equation*}
which proves the mismatch is purely confined to the high-frequency tail $X_N^\perp$.
\newline
Finally, applying operator norm sub-multiplicativity to \eqref{eq:Zmap-bound} and the tail representation of the forcing \eqref{eq:g-is-tail}, we bound the error in $V'$
\begin{align*}
    \|y_\infty(y^\star) - y_r\|_{V'} &= \|(I - P_N)\mathcal{Z}_N(y^\star)\|_{V'} \\
    &\le \|I - P_N\|_{\mathcal{L}(V')} \|\mathcal{Z}_N(y^\star)\|_{V'} \\
    &\le \|I - P_N\|_{\mathcal{L}(V')} C_{\mathrm{cl}} \color{black}{\|(I - P_N)B\|_{\mathcal{L}(\mathbb{R}^M,D(\mathcal A)')}} \|U_N\| \|y^\star\|_{X_N}.
\end{align*}
Substituting $\|y^\star\|_{X_N} \le \|(I+T_N)^{-1}\| \|y_r\|_{X_N}$ establishes the explicit tail estimate, concluding the proof.

\end{proof}

\begin{remark}[Practical assembly of $T_N$ and iterative construction]
\label{rem:picard-small-gain}
To construct the matrix representation of $T_N$ in practice, one simply evaluates the bias on the basis vectors. For each $k=1,\dots,N$, set $y^\star = \varphi_k$, compute $u_r = U_N \varphi_k$, and solve the stationary abstract problem $A_{\mathrm{cl}}z + g(\varphi_k) = 0$ to find $z_\infty(\varphi_k) = \mathcal{Z}_N(\varphi_k)$. The $(\ell, k)$ entry of the $T_N$ matrix is then given by the duality pairing $\langle z_\infty(\varphi_k), \varphi_\ell\rangle_{V',V}$, reducing \eqref{eq:fixed-point} to a straightforward dense linear system solve. 

\medskip
\noindent
Furthermore, if the tail perturbation is small enough such that $\|T_N\|_{\mathcal{L}(X_N)} < 1$, $(I+T_N)$ is guaranteed invertible by a Neumann series. The commanded reference $y^\star$ can alternatively be computed via Picard iteration: $y^{(\ell+1)} = y_r - T_N y^{(\ell)}$ starting from $y^{(0)} = y_r$. From \eqref{eq:Zmap-bound} and \eqref{eq:g-is-tail}, geometric convergence is ensured if the spectral tail $\|(I-P_N)B\|$ is sufficiently small and the feedforward solver $U_N$ is well-conditioned.
\end{remark}

{\color{black}
\section{Thermo-Plasmonic Realization and Proof of Theorem \ref{thm:main-nanoparticle-tracking}}
\label{sec:coupling}
\noindent
We now show how to realize the abstract closed-loop control model using the discrete thermo-plasmonic system. 

\subsection{Explicit rates for the nonzero-initial-data contribution}
\label{subsec:explicit-rates-nonzero-initial}

\noindent
\noindent
To rigorously bound the error contribution of a nonzero initial temperature, we fix a $\delta$-independent bounded neighborhood $\mathcal{U}$ of the particle cluster. {\color{black}Let $S_m(t)$ denote the whole-space heat semigroup associated with the homogeneous background medium, such that $v^0(\cdot,t) := S_m(t)\theta$ solves the unperturbed Cauchy problem $c_m\partial_t v^0 - \gamma_m\Delta v^0 = 0$ in $\mathbb{R}^3$.} For the subsequent estimates, the macroscopically prepared-state class $\mathcal{Y}_T$ is defined as:
\begin{equation}\label{eq:def-YT}
\begin{aligned}
 \mathcal Y_T
 &:=\Bigl\{\theta \in \textcolor{black}{L^2(\mathbb{R}^3)}:\
 S_m(\cdot)\theta\in H^2\bigl(0,T;W^{1,\infty}(\mathcal U)\bigr)\Bigr\},\\
 \|\theta\|_{\mathcal Y_T}
 &:=\|S_m(\cdot)\theta\|_{H^2(0,T;W^{1,\infty}(\mathcal U))}.
\end{aligned}
\end{equation}

\begin{proposition}[Nonzero initial data: explicit $\delta$-rates]
\label{prop:nonzero-initial-reduction}
Assume the hypotheses of Theorem~\ref{thm:effective_dynamics} hold, with the small-gain condition \eqref{cond:invert_heat} holding uniformly in $\delta$. Let $p \in U$ be an admissible illumination, and let $\theta^{\rm in} \in \mathcal{Y}_T$ be an initial state. Let $u^\delta$ and $u^{\rm eff}$ be the physical and effective temperatures, respectively, satisfying $u^\delta(0) = u^{\rm eff}(0) = \theta^{\rm in}$.
\newline
Decomposing the approximation error as $u^\delta - u^{\rm eff} = r_{\rm in}^\delta + r_{\rm FL}^\delta$, the initial-state remainder $r_{\rm in}^\delta$ satisfies the global bounds:
\begin{equation}\label{eq:initial-energy-rate-whole}
 \|r_{\rm in}^\delta\|_{C([0,T];H^{-1}(\mathbb R^3))} + \|r_{\rm in}^\delta\|_{L^2(0,T;L^2(\mathbb R^3))} \le C_T M^{1/2}\delta^{3/2}\|\theta^{\rm in}\|_{\mathcal Y_T},
\end{equation}
\begin{equation}\label{eq:initial-energy-rate}
 \|r_{\rm in}^\delta\|_{C([0,T];V')} + \|r_{\rm in}^\delta\|_{L^2(0,T;H)} \le C_{T,\Omega} M^{1/2}\delta^{3/2}\|\theta^{\rm in}\|_{\mathcal Y_T}.
\end{equation}
On any compact set $B_{\rm far} \Subset \mathbb{R}^3 \setminus \overline{\mathcal{K}}$ separated from the particle cluster, multipole cancellations provide the pointwise bound:
\begin{equation}\label{eq:initial-far-field-rate}
 \|r_{\rm in}^\delta\|_{L^\infty(B_{\rm far}\times[0,T])} \le C_{T,B_{\rm far}} M\delta^3\|\theta^{\rm in}\|_{\mathcal Y_T}.
\end{equation}
Consequently, the explicit initial-data remainder entering the global control-energy estimate is defined by $\eta_{\rm in}^{\rm en}(\delta,M) := M^{1/2}\delta^{3/2}$, whereas the local field approximation yields the improved rate $\eta_{\rm in}^{\rm far}(\delta,M) := M\delta^3$.
\newline
The zero-initial-data remainder $r_{\rm FL}^\delta$ inherits the asymptotic rate $\eta_{\rm src}(\delta)$ from \cite{CaoMukherjeeSini2025}:
\begin{equation}\label{eq:zero-initial-source-rate}
 \|r_{\rm FL}^\delta\|_{C([0,T];V')} + \|r_{\rm FL}^\delta\|_{L^2(0,T;H)} \le C_T\eta_{\rm src}(\delta)\|p\|_U.
\end{equation}
Thus, both approximation errors vanish as $\delta \downarrow 0$, provided the dilute limit condition $M\delta^3 \to 0$ holds.
\end{proposition}

\begin{proof}
\textit{Step 1: Superposition and the Contrast Equation.} \\
Decompose $u^\delta = v^\delta + q^\delta$ and $u^{\rm eff} = v^0 + q^{\rm eff}$, where $v^\delta, v^0$ are the unforced evolutions initialized at $\theta^{\rm in}$, and $q^\delta, q^{\rm eff}$ are the zero-initial-data responses to the illumination. Consequently, $r_{\rm in}^\delta := v^\delta - v^0$ and $r_{\rm FL}^\delta := q^\delta - q^{\rm eff}$. The term $r_{\rm FL}^\delta$ satisfies the (zero-initial-data) bounds derived in \cite{CaoMukherjeeSini2025}.

\medskip
\noindent
Since $v^\delta(0) = v^0(0) = \theta^{\rm in}$, we have $r_{\rm in}^\delta(0) = 0$. The free evolutions satisfy $c_\delta\partial_t v^\delta - \nabla\cdot(\gamma_\delta\nabla v^\delta)=0$ and $c_m\partial_t v^0 - \gamma_m\Delta v^0=0$. Subtracting these yields the contrast equation:
\begin{equation}\label{eq:app-initial-contrast}
 c_m\partial_t r_{\rm in}^\delta - \gamma_m\Delta r_{\rm in}^\delta = \nabla\cdot(\alpha_\delta\nabla v^\delta) - \alpha_\delta^\dagger\partial_t v^\delta,
\end{equation}
where $\alpha_\delta := \gamma_\delta - \gamma_m$ and $\alpha_\delta^\dagger := c_\delta - c_m$ are strictly supported in $D^\delta = \bigcup_{i=1}^M D_i^\delta$.

\medskip\noindent
\textit{Step 2: Lippmann--Schwinger Formulation and Scaled Interior Estimates.} \\
The thermal Lippmann--Schwinger operator governing $v^\delta$ is exactly the one analyzed in \cite[Section~2.1]{CaoMukherjeeSini2025}; only the incident term changes to the macroscopic free evolution. More precisely, $v^\delta$ satisfies:
\begin{equation}\label{eq:LS-nonzero-initial}
 v^\delta + \vartheta V_{D^\delta}^t[\partial_t v^\delta] + \alpha S_{\partial D^\delta}^t[\partial_\nu v^\delta] = S_m(t)\theta^{\rm in}.
\end{equation}
By definition, the prepared incident term satisfies $S_m(t)\theta^{\rm in} \in L^\infty(\mathcal{U})$. Changing variables $x = z_i + \delta\xi$ with volume element $dx = \delta^3 d\xi$, the local mass scales as $\|S_m(t)\theta^{\rm in}\|_{L^2(D_i^\delta)} \le C \delta^{3/2}\|\theta^{\rm in}\|_{\mathcal Y_T}$. 

\medskip
\noindent
Applying the rescaled Fourier--Laplace estimates from \cite[Lemma~2.1]{CaoMukherjeeSini2025} to this specific incident term under the condition \eqref{cond:invert_heat} yields for $k=0,1$:
\[
 \max_{1\le i\le M} \|\partial_t^k v^\delta\|_{L^2(0,T;L^2(D_i^\delta))} \lesssim \delta^{3/2}\|\theta^{\rm in}\|_{\mathcal Y_T},
\]
\[
 \max_{1\le i\le M} \|\partial_t^k\nabla v^\delta\|_{L^2(0,T;L^2(D_i^\delta))} \lesssim \delta^{3/2+\beta}\|\theta^{\rm in}\|_{\mathcal Y_T}.
\]
Consequently, scaling the trace estimates from \cite[Section~2.3]{CaoMukherjeeSini2025}, with the interior PDE $\Delta v^\delta = (c_{p,i}/\gamma_{p,i})\partial_t v^\delta$, provides the associated boundary flux bound:
\begin{equation}\label{eq:app-v-normal}
 \max_{1\le i\le M} \|\partial_t^k\partial_\nu v^\delta\|_{L^2(0,T;L^2(\partial D_i^\delta))} \lesssim \delta^{1+\beta}\|\theta^{\rm in}\|_{\mathcal Y_T}.
\end{equation}

\medskip
\noindent
\textit{Step 3: Global Energy Topology Estimate.} \\
For $\phi \in H^1(\mathbb{R}^3)$, bounding the distributional forcing $f^\delta$ from \eqref{eq:app-initial-contrast} yields:
\[
 |\langle f^\delta, \phi \rangle| \le \sum_{i=1}^M \left( \|\alpha_i \nabla v^\delta\|_{L^2(D_i^\delta)} \|\nabla \phi\|_{L^2(D_i^\delta)} + \|\alpha_i^\dagger \partial_t v^\delta\|_{L^2(D_i^\delta)} \|\phi\|_{L^2(D_i^\delta)} \right).
\]
Since $|\alpha_i| \sim \delta^{-\beta}$ and $|\alpha_i^\dagger| \sim 1$, the dominant gradient term is $\|\alpha_i \nabla v^\delta\|_{L^2(D_i^\delta)} \lesssim \delta^{-\beta} \cdot \delta^{3/2+\beta} = \delta^{3/2}$. Taking the supremum over $\|\phi\|_{H^1(\mathbb{R}^3)} = 1$ yields:
\[
 \|f^\delta\|_{L^2(0,T;H^{-1}(\mathbb R^3))} \le C_T M^{1/2}\delta^{3/2}\|\theta^{\rm in}\|_{\mathcal Y_T}.
\]
Standard coercive parabolic estimates applied to \eqref{eq:app-initial-contrast} immediately yield the whole-space bounds \eqref{eq:initial-energy-rate-whole}. The continuous extension operator from $H^1(\Omega)$ to $H^1(\mathbb{R}^3)$ guarantees the restricted bounds \eqref{eq:initial-energy-rate}.

\medskip\noindent
\textit{Step 4: Far-Field Multipole Cancellation.} \\
For $x \in B_{\rm far}$ strictly separated from $D^\delta$, the representation formula for $r_{\rm in}^\delta$ consists of retarded volume and layer potentials. By Step 2, the volume potentials are bounded by $\mathcal{O}(\delta^3)$ per particle. Expanding the smooth heat kernel $\Phi^{(m)}$ around the center $z_i$ yields the layer potential:
\begin{align*}
 S_{\partial D_i^\delta}^t[\partial_\nu v^\delta](x,t) 
 &= \int_0^t \Phi^{(m)}(x,t;z_i,s) \int_{\partial D_i^\delta} \partial_\nu v^\delta(y,s) \,d\sigma_y \,ds \\
 &\quad + \int_0^t \int_{\partial D_i^\delta} \big(\nabla_y \Phi^{(m)}(z_i) \cdot (y-z_i) + \mathcal{O}(|y-z_i|^2)\big) \partial_\nu v^\delta(y,s) \,d\sigma_y \,ds.
\end{align*}
For the constant monopole term, the Divergence Theorem and interior PDE yield:
\[
 \int_{\partial D_i^\delta}\partial_\nu v^\delta\,d\sigma = \frac{c_{p,i}}{\gamma_{p,i}}\frac{d}{dt} \int_{D_i^\delta}v^\delta\,dx = \mathcal{O}(\delta^\beta \cdot \delta^3) = \mathcal O(\delta^{3+\beta}).
\]
For the dipole term, Cauchy-Schwarz with $\sup |y-z_i| = \mathcal{O}(\delta)$ and the boundary flux bound \eqref{eq:app-v-normal} yields:
\[
 \mathcal{O}\big(\delta\, |\partial D_i^\delta|^{1/2} \|\partial_\nu v^\delta\|_{L^2(\partial D_i^\delta)}\big) = \mathcal{O}(\delta \cdot \delta \cdot \delta^{1+\beta}) = \mathcal O(\delta^{3+\beta}).
\]
Multiplying the total layer potential by the contrast $\alpha_i = \mathcal O(\delta^{-\beta})$ exactly cancels the $\beta$ dependence, contributing identically $\mathcal O(\delta^3)$ per particle. Summing over the $M$ inclusions confirms the pointwise bound \eqref{eq:initial-far-field-rate}. 

\medskip
\noindent
As $\delta \downarrow 0$, the limit $M\delta^3 \to 0$ naturally guarantees the decay.
\end{proof}

\subsection{Abstract reformulation and domain restriction}
\label{subsec:domain-restriction}

\noindent
Recall the singularly actuated parabolic system on the bounded control domain $\Omega$ introduced in Section~\ref{subsec:math-formulation}. To rigorously justify the restriction of the effective whole-space heat propagation (from Theorem~\ref{thm:effective_dynamics}) to this bounded domain, we rely on the following bounds.

\begin{lemma}[Local short-time restriction from $\mathbb R^3$ to $\Omega$]
\label{lem:R3-to-Omega-restriction}
Let $\Omega\subset\mathbb R^3$ be a bounded $C^2$ domain and let $B_{\rm obs}\Subset\Omega$ with
$d_\Omega=\operatorname{dist}(B_{\rm obs},\partial\Omega)>0$. Let
\[
 f(x,t)=\sum_{j=1}^M u_j(t)\delta_{z_j},\qquad z_j\in\Omega,
\]
with $u\in L^1(0,\tau;\mathbb R^M)$. Let $\theta\in L^1(\mathbb R^3)\cap L^2(\mathbb R^3)$ be an initial temperature. Denote by $w$ the whole-space solution
\[
 \partial_t w-\kappa_m\Delta w=f\quad\hbox{in }\mathbb R^3\times(0,\tau),
 \qquad w(\cdot,0)=\theta,
\]
and by $y$ the corresponding Neumann solution in $\Omega$ initialized by the restriction of the same state,
\[
 \partial_t y-\kappa_m\Delta y=f\quad\hbox{in }\Omega\times(0,\tau),\qquad
 \partial_\nu y=0,
 \qquad y(\cdot,0)=\theta|_\Omega.
\]
Then there exists a geometry-dependent $\tau_0>0$ such that, for $0<\tau\le\tau_0$,
\begin{equation}\label{eq:R3-to-Omega-bound}
\sup_{0<t\le\tau}\|w(t)-y(t)\|_{L^2(B_{\rm obs})}
\le C_{\tau_0}\exp\!\left(-\frac{c\,d_\Omega^2}{\kappa_m\tau}\right)
\left(\|\theta\|_{L^1(\mathbb R^3)\cap L^2(\mathbb R^3)}
      +\|u\|_{L^1(0,\tau;\mathbb R^M)}\right).
\end{equation}
The constants are independent of the actuator fill distance and, in particular, the actuator points need not be contained in $B_{\rm obs}$.
\end{lemma}

\begin{proof}
Let $\Phi^{(m)}$ and $G_\Omega^N$ denote the whole-space and Neumann heat kernels. For points $x$ in the fixed interior set $B_{\rm obs}$, the principle of not feeling the boundary gives, uniformly for $z\in\Omega$ and $0<t\le\tau_0$,
\begin{equation}\label{eq:not-feeling-boundary}
 |G_\Omega^N(x,z,t)-\Phi^{(m)}(x,t;z,0)|
 \le C t^{-3/2}\exp\!\left(-\frac{c\,d_\Omega^2}{\kappa_m t}\right).
\end{equation}
The contribution of the initial temperature outside $\Omega$ satisfies the same Gaussian localization estimate because $|x-z|\ge d_\Omega$ for $x\in B_{\rm obs}$ and $z\in\Omega^c$. Applying the two Duhamel formulas, splitting the initial-data integral into $\Omega$ and $\Omega^c$, and using \eqref{eq:not-feeling-boundary} gives
\[
 \|w(t)-y(t)\|_{L^2(B_{\rm obs})}
 \le C\,t^{-3/2}e^{-c d_\Omega^2/(\kappa_m t)}\|\theta\|_{L^1\cap L^2}
 +C\int_0^t (t-s)^{-3/2}e^{-c d_\Omega^2/(\kappa_m(t-s))}\|u(s)\|\,ds.
\]
For $0<t\le\tau\le\tau_0$, the algebraic prefactors are absorbed by decreasing the Gaussian constant $c$, yielding \eqref{eq:R3-to-Omega-bound}.
\end{proof}}

{\color{black}
\medskip
\noindent
The short-time restriction bound in Lemma~\ref{lem:R3-to-Omega-restriction} justifies evaluating the whole-space physical system within the bounded control domain. This approximation relies on the following operational and geometric compatibilities of the closed-loop implementation:
\begin{enumerate}
    \item \textit{Sequential resetting and short time windows.} The (Neumann) truncation cannot uniformly approximate the whole-space heat flow as $t\to\infty$. Therefore, the estimate in Lemma~\ref{lem:R3-to-Omega-restriction} is evaluated locally in space over a short time window $\tau$. In practice, the closed-loop control is updated over consecutive short intervals $[t_n,t_{n+1}]$ of length $\tau$. At each update time $t_n$, the measured physical temperature is used as the initial state for the subsequent forward step. This sequential resetting prevents the open-loop accumulation of truncation errors, restricting the bounded-domain discrepancy to the single-step diffusion scale $\sqrt{\kappa_m\tau}$.

    \item \textit{Temporal support of the optical input.} To maintain a uniform error bound across multiple update steps, the optical illumination on each window can be assigned a short temporal support. Although the heat equation propagates with infinite speed and the Volterra amplitudes retain causal tails, the short optical support ensures a uniform bound on the newly generated heat $\|u\|_{L^1(0,\tau;\mathbb R^M)}$ via the estimates in Theorem~\ref{thm:effective_dynamics}. 

    \item \textit{Decoupling of geometric requirements.} The geometric requirements for stabilization and approximation are fundamentally decoupled. Stabilization requires a sufficiently small actuator fill distance $h_M$ (Assumption~\ref{assump:KRW-act}), whereas the bounded-domain approximation only requires the observation region to be separated from the boundary by a few short-window diffusion lengths ($\sqrt{\kappa_m\tau}\ll d_\Omega$). As a result, the control domain $\Omega$ does not need to scale with the total treatment time; a sequence of short feedback updates allows both $\Omega$ and the chosen actuator configuration to remain fixed.
\end{enumerate}
}

{\color{black}
\medskip
\noindent
For a finite horizon, let $0=t_0<\cdots<t_K=T$ with $\tau_n:=t_{n+1}-t_n\le\tau_0$. Apply Lemma~\ref{lem:R3-to-Omega-restriction} on each $I_n=[t_n,t_{n+1}]$, using the measured physical state at $t_n$ as common initial datum, and denote its right-hand side by $\eta_{\Omega,n}$. Then
\begin{equation}\label{eq:cumulative-domain-error}
\|w-y\|_{L^2(0,T;L^2(B_{\rm obs}))}^2
\le \sum_{n=0}^{K-1}\tau_n\eta_{\Omega,n}^2
\le T\eta_{\Omega,T}^2,
\qquad
\eta_{\Omega,T}:=\max_{0\le n<K}\eta_{\Omega,n}.
\end{equation}
In particular, introduce the dimensionless boundary-separation parameter
$\rho_{\Omega,\tau_0}:=d_\Omega/\sqrt{\kappa_m\tau_0}$. Since $\tau_n\le\tau_0$, Lemma~\ref{lem:R3-to-Omega-restriction} gives, writing $\theta_n$ and $u_n$ for the reset initial state and input on $I_n$,
\[
\eta_{\Omega,T}\le C_{\tau_0}e^{-c\rho_{\Omega,\tau_0}^2}
\max_{0\le n<K}\left(
\|\theta_n\|_{L^1(\mathbb R^3)\cap L^2(\mathbb R^3)}
+\|u_n\|_{L^1(I_n;\mathbb R^M)}\right).
\]
Thus, provided these windowwise norms remain uniformly bounded, $\rho_{\Omega,\tau_0}\gg1$ makes $\eta_{\Omega,T}$ exponentially small, and this smallness is inherited directly by the tracking estimates containing $\eta_{\Omega,T}$. The local estimate therefore extends over any finite sequence of update windows, with the thermal history retained in the measured initial states.
}

\subsection{Abstract Map from Intensity Controls to Heat Inputs} 
\label{sec:map-EM-Heat}

\noindent
In the thermo-plasmonic model, the effective volumetric source is parameterized as a non-negative linear combination of precomputed intensity patterns:
\begin{equation}
\label{eq:heat-source-dictionary}
    q(\cdot,t) = \sum_{\ell=1}^P p_\ell(t)\,\psi_\ell(\cdot), \qquad \psi_\ell(\cdot) = \bigl|E^{(\ell)}_{\mathrm{in}}(\cdot)\bigr|^2.
\end{equation}
The actuation map is therefore strictly linear with respect to the intensity vector $p(t) := (p_\ell(t))_{\ell=1}^P \in \mathbb{R}^P$.

\medskip
\noindent
Using the temporal spaces $U, X,$ and $Y$ defined in Section \ref{subsec:functional-framework}, the Volterra system governing the nanoparticle amplitudes admits the operator form
\begin{equation}
    \sigma = \mathcal{V}_\delta F + r,
\end{equation}
where $F=(F_i)$ is the forcing vector depending linearly on the intensity history $p$, $\mathcal V_\delta$ is the bounded Volterra solution operator on $X$, and $r$ is an asymptotic remainder. The effective heat inputs generated by the nanoparticles are defined as
\begin{equation}
\label{eq:Gi-sigma}
    G_i(t) := \textcolor{black}{-} \frac{\alpha_i}{c_m}\,\operatorname{vol}(D^\delta_i)\,\sigma^{\textcolor{black}{(i)}}(t), \qquad i=1,\dots,M,
\end{equation}
yielding the vector $G = (G_i)_{i=1}^M \in Y$.

\begin{proposition}[Approximation of the map $p\mapsto G$]
\label{thm:K-rigorous}
The following operators are bounded:
\begin{enumerate}
    \item[(i)] The map $p\mapsto F(p)$ extends to a bounded linear operator
    \[
      \mathcal T_\delta\in\mathcal L(U,X).
    \]
    \item[(ii)] The Volterra system defines the boundedly invertible solution operator
    $\mathcal V_\delta=(I+\mathcal W_\delta)^{-1}\in\mathcal L(X,X)$, and, for some $\mu>0$,
    \[
      \|r\|_X\le C_T M\delta^\mu\|F\|_X,
    \]
    where $C_T>0$ is independent of $\delta$.
    \item[(iii)] The scaling map $\sigma \mapsto G$ in \eqref{eq:Gi-sigma} extends to a bounded linear operator $\mathcal S_\delta\in\mathcal L(X,Y)$.
\end{enumerate}
In addition, the effective heat input satisfies the decomposition
\begin{equation}
\label{eq:K-decomposition}
\begin{aligned}
 G&=\mathcal K_\delta p+\mathcal R^\delta p,\\
 \mathcal K_\delta&:=\mathcal S_\delta\mathcal V_\delta\mathcal T_\delta
 \in\mathcal L(U,Y).
\end{aligned}
\end{equation}
Furthermore, the remainder operator $\mathcal{R}^\delta \in \mathcal{L}(U, Y)$ obeys the uniform asymptotic bound
\begin{equation}
\label{eq:Rdelta-bound}
    \|\mathcal{R}^\delta p\|_Y \le C_T' M \delta^\mu \|p\|_U,
\end{equation}
for a constant $C_T' > 0$ independent of $\delta$.
\end{proposition}

\begin{proof}
\noindent (i) follows from the linear dependence of $F_i$ on the illumination intensities. Condition (ii) follows from the invertibility of $(I+\mathcal W_\delta)$ in the Volterra system and the explicit error estimates derived in \cite{CaoMukherjeeSini2025}. Condition (iii) is immediate from \eqref{eq:Gi-sigma} and the continuous embedding $H^1(0,T)\hookrightarrow L^2(0,T)$. For the fixed particle configuration, we make the scale dependence explicit by writing $\mathcal T_\delta$, $\mathcal V_\delta$, and $\mathcal S_\delta$.  Thus $\sigma=\mathcal V_\delta\mathcal T_\delta p+r$ and $G=\mathcal S_\delta\sigma=\mathcal S_\delta\mathcal V_\delta\mathcal T_\delta p+\mathcal S_\delta r$.  Setting $\mathcal K_\delta:=\mathcal S_\delta\mathcal V_\delta\mathcal T_\delta$ and $\mathcal R^\delta p:=\mathcal S_\delta r$ yields \eqref{eq:K-decomposition}. The bound \eqref{eq:Rdelta-bound} follows from the boundedness of $\mathcal S_\delta$ and $\mathcal T_\delta$, and the assumed bound on $\|r\|_X$.
\end{proof}

\noindent
Physical power intensities require $p_\ell(t)\ge0$ for almost every $t$.  For the spaces $U_0$ and
$Y_0$ constructed in Subsection~\ref{sec:K-inverse}, define the closed convex cone
\[
 U_{0,+}:=\{p\in U_0:p_\ell(t)\ge0\ \text{for a.e. }t,\ 1\le\ell\le P\},
 \qquad
 \mathcal C_{\delta,Q}:=\mathcal K_\delta(U_{0,+})\subset Y_0.
\]
Because $\mathcal K_\delta$ contains a Volterra operator, feasibility and approximation are
whole-history problems.  If $u_{\mathrm{des}}\in\mathcal C_{\delta,Q}$, a nonnegative realization may
be computed from
\[
 \min\{\|p\|_U^2:p\in U_{0,+},\ \mathcal K_\delta p=u_{\mathrm{des}}\};
\]
otherwise one solves
\[
 \min_{p\in U_{0,+}}\|\mathcal K_\delta p-u_{\mathrm{des}}\|_Y^2.
\]
The linear tracking theorem below uses the unconstrained minimum-norm right-inverse on $Y_0$ and
therefore permits signed reduced coefficients.  Incorporating the nonnegativity cone directly into
the closed-loop analysis remains an implementation constraint, as stated in the conclusion.

\subsection{Finite-Dimensional Right-Inverse on the Reachable Heat-Input Space}
\label{sec:K-inverse}

\noindent
Let
\[
 \Theta_Q:=\operatorname{span}\{\vartheta_1,\ldots,\vartheta_Q\}
 \subset H^{r+1/2}_{0,\sigma}(0,T)
\]
be a finite-dimensional space of linearly independent causal temporal profiles, and let
$e_1,\ldots,e_P$ be the canonical basis of $\mathbb R^P$.  We choose the implementable
illumination space
\begin{equation}\label{eq:U0-temporal-dictionary}
 U_0:=\operatorname{span}\{e_\ell\vartheta_q:
 1\le\ell\le P,\ 1\le q\le Q\}\subset U,
 \qquad \dim U_0=PQ.
\end{equation}
Denote by
\[
 \mathcal I_{U,Q}:\mathbb R^{PQ}\longrightarrow U_0
\]
the coefficient isomorphism associated with this dictionary.  The finite-dimensional realization
operator is
\begin{equation}\label{eq:finite-realization-operator}
 \mathbf K_{\delta,Q}:=\mathcal K_\delta\mathcal I_{U,Q}:
 \mathbb R^{PQ}\longrightarrow Y,
\end{equation}
and the realizable effective-input space is defined by the actual output histories
\begin{equation}\label{eq:Y0-reachable-range}
\quad Y_0=Y_0^\delta:=\operatorname{Ran}\mathbf K_{\delta,Q}
       =\mathcal K_\delta(U_0)\subset Y.\quad
\end{equation}
In particular, the Volterra operator is not required to preserve $\Theta_Q$.  The superscript
$\delta$ on $Y_0$ is suppressed for the fixed configuration considered below.

\begin{proposition}[Right-inverse on the reachable heat-input space]
\label{prop:K-right-inverse}
Suppose Assumption~\ref{ass:illumination} and the Volterra small-gain condition hold.  Then, for
$\delta>0$ sufficiently small,
\begin{equation}\label{eq:finite-realization-rank}
 \operatorname{rank}\mathbf K_{\delta,Q}=MQ,
 \qquad \dim Y_0=MQ.
\end{equation}
Let $\mathbf K_{\delta,Q}^{\dagger}$ denote the Moore--Penrose inverse on
$(\ker\mathbf K_{\delta,Q})^\perp$, and define
\begin{equation}\label{eq:operator-right-inverse}
 \mathcal K_{\delta,Q}^{\dagger}
 :=\mathcal I_{U,Q}\mathbf K_{\delta,Q}^{\dagger}:Y_0\longrightarrow U_0.
\end{equation}
Then
\begin{equation}\label{eq:exact-right-inverse-reachable}
 \mathcal K_\delta\mathcal K_{\delta,Q}^{\dagger}=I_{Y_0}.
\end{equation}
If
\begin{equation}\label{eq:actuation-singular-value}
 c_{\rm act}(\delta,Q)
 :=\inf_{\substack{a\perp\ker\mathbf K_{\delta,Q}\\ |a|=1}}
     \|\mathbf K_{\delta,Q}a\|_Y>0,
\end{equation}
then
\begin{equation}\label{eq:Cinv-definition}
 C_{\rm inv}(\delta,Q)
 :=\|\mathcal K_{\delta,Q}^{\dagger}\|_{\mathcal L(Y_0,U_0)}
 \le \frac{\|\mathcal I_{U,Q}\|}{c_{\rm act}(\delta,Q)}.
\end{equation}
Consequently, rank is sufficient for a bounded right-inverse at each fixed $\delta$.  A bound
uniform in $\delta$ follows if $c_{\rm act}(\delta,Q)\ge c_*>0$ and the temporal dictionary is
uniformly normalized.
\end{proposition}

\begin{proof}
On $U_0$, the instantaneous forcing map has the form
\[
 (\mathcal T_\delta p)(t)=L_\delta p(t),
 \qquad L_\delta=(c_i^{(\ell)}(\delta))\in\mathbb R^{M\times P}.
\]
Assumption~\ref{ass:illumination} implies $\operatorname{rank}L_\delta=M$ for sufficiently small
$\delta$.  Hence
\[
 \mathcal T_\delta(U_0)=\mathbb R^M\otimes\Theta_Q,
 \qquad \dim\mathcal T_\delta(U_0)=MQ.
\]
The Volterra solution operator $\mathcal V_\delta=(I+\mathcal W_\delta)^{-1}$ is injective under the
small-gain condition, and the diagonal scaling map $\mathcal S_\delta$ is injective for every fixed
$\delta>0$.  Therefore
\[
 \operatorname{rank}(\mathcal S_\delta\mathcal V_\delta
 \mathcal T_\delta|_{U_0})=MQ,
\]
which proves \eqref{eq:finite-realization-rank}.  By definition, $\mathbf K_{\delta,Q}$ is
surjective onto $Y_0$, and its Moore--Penrose inverse satisfies
$\mathbf K_{\delta,Q}\mathbf K_{\delta,Q}^{\dagger}=I_{Y_0}$.  Consequently,
\[
 \mathcal K_\delta\mathcal K_{\delta,Q}^{\dagger}
 =\mathcal K_\delta\mathcal I_{U,Q}\mathbf K_{\delta,Q}^{\dagger}
 =\mathbf K_{\delta,Q}\mathbf K_{\delta,Q}^{\dagger}=I_{Y_0}.
\]
Finally, $\|\mathbf K_{\delta,Q}^{\dagger}\|=c_{\rm act}(\delta,Q)^{-1}$, which gives
\eqref{eq:Cinv-definition}.
\end{proof}

\begin{remark}[One temporal profile]
For $Q=1$, one may take $U_0=\{\alpha\vartheta_1:\alpha\in\mathbb R^P\}$, and
$\dim Y_0=M$.  However,
\[
 Y_0=\mathcal K_\delta(U_0)
\]
is generally not the prescribed space
$\{\beta\vartheta_1:\beta\in\mathbb R^M\}$: the Volterra dynamics generate $M$ independent
vector-valued output histories whose temporal shapes need not be proportional to $\vartheta_1$.
Using $Q>1$ supplies a hierarchy of richer reachable spaces and can reduce the projection error in
Subsection~\ref{subsec:projection_Y0}.
\end{remark}

\subsection{Proof of Theorem \ref{thm:main-nanoparticle-tracking}}

\noindent
We now possess all mathematical components to prove the final physical tracking result.

\begin{proof}[Proof of Theorem \ref{thm:main-nanoparticle-tracking}]
For $u_{\mathrm{des}}\in Y_0$, choose
$p=\mathcal K_{\delta,Q}^{\dagger}u_{\mathrm{des}}$. Let $y_{\mathrm{eff}}$ denote the
whole-space homogeneous-background effective temperature with the same initial datum as the
physical nanoparticle model and with effective point-source amplitudes $G$. We split
\[
 y_{\mathrm{phys}}-y_{Y_0}
 =(y_{\mathrm{phys}}-y_{\mathrm{eff}})
 \textcolor{black}{+(y_{\mathrm{eff}}-w_{Y_0})}
 \textcolor{black}{+(w_{Y_0}-y_{Y_0}),}
\]
\textcolor{black}{where $w_{Y_0}$ is the whole-space homogeneous-background solution driven by the desired control $u_{\mathrm{des}}$.}
By Proposition~\ref{prop:nonzero-initial-reduction} and
\eqref{eq:zero-initial-source-rate},
\begin{equation}\label{eq:proof-initial-model-bound}
 \|y_{\mathrm{phys}}-y_{\mathrm{eff}}\|_{C([0,T];V')}
 +\|y_{\mathrm{phys}}-y_{\mathrm{eff}}\|_{L^2(0,T;H)}
 \le C_T\Bigl(
 M^{1/2}\delta^{3/2}\|y_0\|_{\mathcal Y_T}
 +\eta_{\rm src}(\delta)\|p\|_U\Bigr).
\end{equation}
\textcolor{black}{The term $y_{\mathrm{eff}}-w_{Y_0}$ represents the discrepancy due to the difference between the actual effective heat input $G$ and the desired abstract input $u_{\mathrm{des}}$.}
Proposition~\ref{thm:K-rigorous} and
\eqref{eq:exact-right-inverse-reachable} give
\[
 G=\mathcal K_\delta p+\mathcal R^\delta p
   =u_{\mathrm{des}}+\mathcal R^\delta p.
\]
Thus, with $\rho:=G-u_{\mathrm{des}}$, \eqref{eq:Rdelta-bound} and
\eqref{eq:Cinv-definition} imply
\begin{equation}\label{eq:proof-rho-bound}
 \|\rho\|_Y
 \le C_T M\delta^\mu C_{\rm inv}(\delta,Q)
      \|u_{\mathrm{des}}\|_Y.
\end{equation}
Moreover,
\[
 \|p\|_U\le C_{\rm inv}(\delta,Q)\|u_{\mathrm{des}}\|_Y.
\]
\noindent
{\color{black}
Both the zero-initial modeling error in \eqref{eq:proof-initial-model-bound} and the actuation remainder in \eqref{eq:proof-rho-bound} are thus bounded by \textcolor{black}{the combined conditioned remainder} 
\textcolor{black}{\begin{equation}\label{eq:eta-FL-conditioned}
    \eta_{\rm FL}(\delta,Q) := C_{\rm inv}(\delta,Q) \big(\eta_{\rm src}(\delta) + C_T M \delta^\mu\big),
\end{equation}}
\textcolor{black}{scaling the total input norm} $\|u_{\mathrm{des}}\|_Y$. 
\newline
To restrict the error estimates to the observation region, we use the embedding $\| \cdot \|_{L^2(B_{\rm obs})} \le \| \cdot \|_H$. Omitting the non-negative $V'$ term in \eqref{eq:proof-initial-model-bound} and applying this embedding yields:
\begin{equation}
    \|y_{\mathrm{phys}} - y_{\mathrm{eff}}\|_{L^2(0,T; L^2(B_{\rm obs}))}^2 \le C_T \Big( M \delta^3 \|y_0\|_{\mathcal Y_T}^2 + \eta_{\rm FL}(\delta,Q)^2 \|u_{\mathrm{des}}\|_Y^2 \Big).
\end{equation}
The difference $y_{\mathrm{eff}}-w_{Y_0}$ solves the whole-space homogeneous-background equation with zero initial difference and forcing $B\rho$. Proposition~\ref{prop:wp-Dirac} combined with the spatial embedding yields the bound:
\begin{equation}
    \|y_{\mathrm{eff}} - w_{Y_0}\|_{L^2(0,T; L^2(B_{\rm obs}))}^2 \le C_T \|\rho\|_Y^2 \le C_T \eta_{\rm FL}(\delta,Q)^2 \|u_{\mathrm{des}}\|_Y^2.
\end{equation}
{\color{black}The difference $w_{Y_0}-y_{Y_0}$ is the bounded-domain restriction error. For a single short window $T\le\tau_0$, Lemma~\ref{lem:R3-to-Omega-restriction} applies directly; for a general finite treatment horizon, we partition $[0,T]$ into consecutive update windows of length at most $\tau_0$ and use \eqref{eq:cumulative-domain-error}. Hence
\begin{equation}
    \|w_{Y_0} - y_{Y_0}\|_{L^2(0,T; L^2(B_{\rm obs}))}^2 \le T\eta_{\Omega,T}^2.
\end{equation}}
Applying the triangle inequality in $L^2(B_{\rm obs})$ to the three terms, and integrating over time yields the estimate
\begin{align}
    \|y_{\mathrm{phys}} - y_{Y_0}\|_{L^2(0,T; L^2(B_{\rm obs}))}^2 
    &\le 3 \Big( \|y_{\mathrm{phys}} - y_{\mathrm{eff}}\|^2 + \|y_{\mathrm{eff}} - w_{Y_0}\|^2 + \|w_{Y_0} - y_{Y_0}\|^2 \Big) \nonumber \\
    &\le C_T' \Big( M\delta^3\|y_0\|_{\mathcal Y_T}^2 + \eta_{\rm FL}(\delta,Q)^2\|u_{\mathrm{des}}\|_Y^2 + \eta_{\Omega,T}^2 \Big),
\end{align}
which proves \eqref{eq:physical-tracking-nonzero-initial}. A similar application of the triangle inequality in $L^2(B_{\rm obs})$ combined with Corollary~\ref{cor:L2-after-t0} establishes the positive-time estimate, concluding the proof. \hfill $\blacksquare$
}
\end{proof}

\subsection{Robustness: projection and thermo--plasmonic realization errors}
\label{sec:robust-Vprime}

\noindent
The ideal closed-loop error $z_{\mathrm{ideal}}=y_{\mathrm{ideal}}-y_r$ is generated by the feedback
\eqref{eq:KRW-feedback} and solves \eqref{eq:KRW-closed-loop} with the constant forcing $g$ given by
\eqref{eq:feedforward-g}. Let $z_\infty$ be the corresponding equilibrium defined in \eqref{eq:z-infty}.
Then the shifted error $\tilde z_{\mathrm{ideal}}:=z_{\mathrm{ideal}}-z_\infty$ satisfies the homogeneous
stabilized dynamics and hence the exponential estimate in $V'$:
\[
\|\tilde z_{\mathrm{ideal}}(t)\|_{V'} \le e^{-\mu t}\|\tilde z_{\mathrm{ideal}}(0)\|_{V'},\qquad t\ge 0,
\]
for suitable $M,\lambda$ as in Theorem~\ref{thm:main-stabilization}.

\paragraph{Finite-horizon deviation estimates.} Let $z$ be the actual error produced by the thermo--plasmonic actuation scheme.
Assume that the actual forcing differs from the ideal Dirac forcing by a remainder
\color{black}{$r(t)\in D(\mathcal A)'$} on $[0,T]$, i.e.
\[
\dot z(t)=A_0 z(t) -\lambda\sum_{j=1}^M \langle d_j,\mathcal A^{-1}z(t)\rangle\, d_j + g + r(t),
\quad z(0)=z_0\in V'.
\]
Then, by the well-posedness estimate in Proposition~\ref{prop:wp-Dirac} applied to the difference
$e(t):=z(t)-z_{\mathrm{ideal}}(t)$, one obtains
\begin{equation}
\label{eq:err-Vprime}
\|e\|_{C([0,T];V')}^2 + \|e\|_{L^2(0,T;H)}^2
\le C_T\,\color{black}{\|r\|_{L^2(0,T;D(\mathcal A)')}^2}.
\end{equation}

\paragraph{$L^2$ tracking after a positive time.} Fix $t_0\in(0,T]$. Combining \eqref{eq:err-Vprime} with Corollary~\ref{cor:L2-after-t0},
we obtain the $L^2$-in-space deviation estimate on $[t_0,T]$:
\[
\sup_{t\in[t_0,T]}\|z(t)-z_{\mathrm{ideal}}(t)\|_{L^2(\Omega)}
\le C(t_0)\,\|e\|_{L^2((t_0/2,T);H)}
\le C(t_0)\,C_T^{1/2}\,\color{black}{\|r\|_{L^2(0,T;D(\mathcal A)')}}.
\]
This is the form in which we will use the robustness estimate in the remainder of the paper.

\subsection{Projection and Total Tracking Error}\label{subsec:projection_Y0}

\noindent
Let $u_{\mathrm{ideal}}$ denote the ideal Dirac-actuator input generated by the feedback (Section~\ref{sec:stab-KRW}). In general $u_{\mathrm{ideal}}$ lies in $L^2(0,T;\mathbb{R}^M)$, not necessarily in the finite-dimensional subspace $Y_0$ that can be realized efficiently by the thermo--plasmonic actuators.
\newline
A natural way to enforce the constraint $u_{\mathrm{des}}\in Y_0$ is to project $u_{\mathrm{ideal}}$ onto $Y_0$ and to use only this projected signal as the ``desired'' input for the thermo--plasmonic realization. More precisely, we set
\begin{equation}
\label{eq:5.Y0-u-des}
   u_{\mathrm{des}} := P_{Y_0} u_{\mathrm{ideal}},
\end{equation}
where $P_{Y_0}$ denotes the orthogonal projection from $L^2(0,T;\mathbb{R}^M)$ onto $Y_0$. By construction, $u_{\mathrm{des}}\in Y_0$ and therefore satisfies the realizability hypothesis of Theorem~\ref{thm:main-nanoparticle-tracking}.
\newline
Let $y_{Y_0}$ denote the solution of the heat equation driven by $u_{\mathrm{des}}$:
\begin{equation}
\label{eq:5.Y0}
\begin{cases}
\dot y_{Y_0}(t) = A_0 y_{Y_0}(t) + B u_{\mathrm{des}}(t), & t>0,\\[0.2em]
y_{Y_0}(0) = y_0.
\end{cases}
\end{equation}
The next proposition quantifies the discrepancy between $y_{Y_0}$ and the ideal closed-loop trajectory $y_{\mathrm{ideal}}$.

\begin{proposition}[Effect of restricting the control to $Y_0$ (in the $V'$ framework)]
\label{prop:projection-error}
Let $T>0$ and let $u_{\mathrm{ideal}}\in L^2(0,T;\R^M)$. Define
$u_{\mathrm{des}}:=P_{Y_0}u_{\mathrm{ideal}}$ as in \eqref{eq:5.Y0-u-des}.
Let $y_{\mathrm{ideal}}$ be the solution of \eqref{eq:heat-DAdual} driven by $u_{\mathrm{ideal}}$
and let $y_{Y_0}$ be the solution of \eqref{eq:5.Y0} driven by $u_{\mathrm{des}}$, with the same initial
condition $y_0\in H\subset V'$. Then there exists $C_T>0$ such that
\begin{equation}
\label{eq:proj-error}
\| y_{Y_0} - y_{\mathrm{ideal}} \|_{C([0,T];V')}^2
+ \| y_{Y_0} - y_{\mathrm{ideal}} \|_{L^2(0,T;H)}^2
\le C_T\,\|(I-P_{Y_0})u_{\mathrm{ideal}}\|_{L^2(0,T;\R^M)}^2.
\end{equation}
Moreover, for every $t_0\in(0,T]$ there exists $C_{T,t_0}>0$ such that
\begin{equation}
\label{eq:proj-error-H-after-t0}
\sup_{t\in[t_0,T]} \| y_{Y_0}(t) - y_{\mathrm{ideal}}(t) \|_{H}
\le C_{T,t_0}\,\|(I-P_{Y_0})u_{\mathrm{ideal}}\|_{L^2(0,T;\R^M)}.
\end{equation}
In particular,
\begin{equation}
\label{eq:best-approx}
\| y_{Y_0} - y_{\mathrm{ideal}} \|_{C([0,T];V')}
\le C_T^{1/2}\inf_{v\in Y_0}\|u_{\mathrm{ideal}}-v\|_{L^2(0,T;\R^M)}.
\end{equation}
\end{proposition}

\begin{proof}
Set $e_{\mathrm{proj}}:=y_{Y_0}-y_{\mathrm{ideal}}$. Then
\[
\dot e_{\mathrm{proj}}(t) = A_0 e_{\mathrm{proj}}(t) + B\bigl(u_{\mathrm{des}}(t)-u_{\mathrm{ideal}}(t)\bigr)
= A_0 e_{\mathrm{proj}}(t) - B\,(I-P_{Y_0})u_{\mathrm{ideal}}(t),
\qquad e_{\mathrm{proj}}(0)=0,
\]
\color{black}{in $D(\mathcal A)'$.} Applying Proposition~\ref{prop:wp-Dirac} yields \eqref{eq:proj-error}.
The estimate \eqref{eq:proj-error-H-after-t0} follows from \eqref{eq:proj-error} and
Corollary~\ref{cor:L2-after-t0}. Finally, since $P_{Y_0}$ is the orthogonal projector onto $Y_0$,
\[
\|(I-P_{Y_0})u_{\mathrm{ideal}}\|_{L^2} = \inf_{v\in Y_0}\|u_{\mathrm{ideal}}-v\|_{L^2},
\]
which gives \eqref{eq:best-approx}.
\end{proof}

\medskip
\noindent
Theorem~\ref{thm:main-nanoparticle-tracking} then compares $y_{Y_0}$ with the physical temperature $y_{\mathrm{phys}}$ generated by the thermo--plasmonic actuators when the electromagnetic control is chosen as $p=\mathcal K_{\delta,Q}^{\dagger}u_{\mathrm{des}}$. Combining Proposition~\ref{prop:projection-error} with Theorem~\ref{thm:main-nanoparticle-tracking} yields an explicit bound on the total tracking error between the ideal closed-loop trajectory and the physically realized one.

\begin{corollary}[Total tracking error: projection + nanoparticles (in $B_{\rm obs}$)]
\label{cor:total-error}
Let $T>0$ and let $u_{\mathrm{ideal}}\in L^2(0,T;\R^M)$ be an ``ideal'' heat input.
Set $u_{\mathrm{des}}:=P_{Y_0}u_{\mathrm{ideal}}$ and let $y_{\mathrm{ideal}}$, $y_{Y_0}$ be as in
Proposition~\ref{prop:projection-error}. Let $y_{\mathrm{phys}}$ be the physical temperature produced by
the thermo--plasmonic actuators when the intensity control $p$ is chosen as in Theorem~\ref{thm:main-nanoparticle-tracking}:
\[
   p=\mathcal K_{\delta,Q}^{\dagger}u_{\mathrm{des}}.
\]
\textcolor{spacecorrectgreen}{Let $\eta_{\Omega,T}$ denote the bounded-domain restriction error on $B_{\rm obs}$, as specified in Theorem~\ref{thm:main-nanoparticle-tracking}.} Then there exists $C_T>0$ such that
\begin{equation}
\label{eq:total-error}
\begin{aligned}
\|y_{\mathrm{phys}}-y_{\mathrm{ideal}}\|_{\textcolor{black}{L^2(0,T;L^2(B_{\rm obs}))}}^2 
&\le C_T\Bigl(
\|(I-P_{Y_0})u_{\mathrm{ideal}}\|_{L^2(0,T;\R^M)}^2 \\
&\quad+ M\delta^3\,\|y_0\|_{\mathcal Y_T}^2
+ \eta_{\rm FL}(\delta,Q)^2\,\|u_{\mathrm{des}}\|_{L^2(0,T;\R^M)}^2
\textcolor{spacecorrectgreen}{+ \eta_{\Omega,T}^2}
\Bigr),
\end{aligned}
\end{equation}
where the explicit initial-state factor $M^{1/2}\delta^{3/2}$ and the conditioned thermo--plasmonic realization remainder $\eta_{\rm FL}(\delta,Q)$ are as in Theorem~\ref{thm:main-nanoparticle-tracking}. Moreover, for every $t_0\in(0,T]$,
\begin{align*}
\sup_{t\in[t_0,T]}\|y_{\mathrm{phys}}(t)-y_{\mathrm{ideal}}(t)\|_{\textcolor{black}{L^2(B_{\rm obs})}}
&\le C_{T,t_0}\Bigl(
\|(I-P_{Y_0})u_{\mathrm{ideal}}\|_{L^2(0,T;\R^M)} \\
&\quad+ M^{1/2}\delta^{3/2}\,\|y_0\|_{\mathcal Y_T}
+ \eta_{\rm FL}(\delta,Q)\,\|u_{\mathrm{des}}\|_{L^2(0,T;\R^M)}
\textcolor{spacecorrectgreen}{+ \eta_{\Omega,T}}
\Bigr),
\end{align*}
for some constant $C_{T,t_0}>0$.
\end{corollary}

\begin{proof}
We decompose
\[
y_{\mathrm{phys}}-y_{\mathrm{ideal}}=(y_{\mathrm{phys}}-y_{Y_0})+(y_{Y_0}-y_{\mathrm{ideal}}).
\]
Applying Theorem~\ref{thm:main-nanoparticle-tracking} to $y_{\mathrm{phys}}-y_{Y_0}$ \textcolor{black}{on the observation region $B_{\rm obs}$} and Proposition~\ref{prop:projection-error}
to $y_{Y_0}-y_{\mathrm{ideal}}$, and then using the triangle inequality, yields \eqref{eq:total-error}.
The positive-time estimate on $[t_0,T]$ follows similarly from the positive-time estimate in Theorem~\ref{thm:main-nanoparticle-tracking} and \eqref{eq:proj-error-H-after-t0}.
\end{proof}

\medskip

\noindent
Proposition~\ref{prop:projection-error} measures the error introduced by restricting the abstract feedback design to controls lying in the finite-dimensional subspace $Y_0$, while Theorem~\ref{thm:main-nanoparticle-tracking} measures the additional error incurred when realizing such controls by thermo--plasmonic nanoparticles. Corollary~\ref{cor:total-error} shows that the total tracking error is the sum of these two contributions: a purely ``approximation'' term, given by the best $L^2$--approximation of $u_{\mathrm{ideal}}$ in $Y_0$, and a physical-realization contribution consisting of the nonzero-initial-data remainder $M^{1/2}\delta^{3/2}\|y_0\|_{\mathcal Y_T}$, \textcolor{spacecorrectgreen}{the bounded-domain restriction error $\eta_{\Omega,T}$,} and the conditioned Foldy--Lax/actuation remainder $\eta_{\rm FL}(\delta,Q)\|u_{\mathrm{des}}\|$.

\section{Algorithmic design} \label{sec:algorithm}
\noindent
In this section, we summarize how the abstract results of Sections~\ref{sec:feedback-proof}--\ref{sec:coupling} can be turned into a constructive procedure for heat tracking using thermo--plasmonic nanoparticles.
\newline
The feedback design in Section~\ref{sec:stab-KRW} produces an \emph{ideal} closed-loop
trajectory $y_{\mathrm{ideal}}$ for the abstract point-actuator model, driven
by an implicit control signal $u_{\mathrm{ideal}}$. The thermo--plasmonic
analysis in Section~\ref{sec:coupling} shows how to realize a finite-dimensional approximation
$u_{\mathrm{des}}\in Y_0$ of this signal by suitable EM illuminations
$p\in U_0$, with an error controlled by the particle size $\delta$ and the calibrated
inverse-conditioning factor $C_{\rm inv}(\delta,Q)$.
\newline
We now discuss how each ingredient can be constructed in practice and how
they combine into an algorithm.

\subsection{Implementable feedback law}
\label{subsec:alg-KRW}

\noindent
Let $y_r\in X_N$ satisfy \eqref{eq:ref-finite-modes}. For equilibrium references (i.e.\ $A_0y_r=0$), we set
$\bar y:=y_r$ and $u_r:=0$. For general (non-equilibrium) references, we incorporate the fixed-point
pre-compensation of Subsection~\ref{sec:fixed-point-bias}: we compute the \emph{commanded reference}
$y^\star\in X_N$ as the solution of \eqref{eq:fixed-point} (see Proposition~\ref{cor:fixed-point-tracking}) and set
$\bar y:=y^\star$. In both cases, we compute a constant feedforward vector $u_r\in\R^M$ associated with
$\bar y$ by solving \eqref{eq:feedforward-linear-system} (equivalently $u_r=U_N\bar y$ in the notation of Section~\ref{sec:stab-KRW}).
\newline
We use the feedback
\begin{equation}
\label{eq:alg-u-ideal}
u_{\mathrm{ideal},j}(t) = (u_r)_j-\lambda\,\big(\mathcal A^{-1}(y(t)-\bar y)\big)(z_j),
\qquad j=1,\dots,M, \quad \mathcal A := I-\kappa_m\Delta,
\end{equation}
which is equivalent to \eqref{eq:KRW-feedback} (with reference $\bar y$).
In the fixed-point compensated case, the reached steady profile $y_\infty(\bar y)$ satisfies
$P_Ny_\infty(\bar y)=y_r$ by \eqref{eq:exact-low-mode-tracking}, and the remaining mismatch is a pure tail
controlled by Corollary~\ref{cor:fixed-point-tracking}.

\paragraph{Offline choice of actuator points and $\lambda$.}
Choose actuator locations $\textcolor{black}{z_1},\dots,\textcolor{black}{z_M}$ so that Assumption~\ref{assump:KRW-act} holds.
In rectangular domains, \cite{KunRodWal24-cocv} provides explicit constructive placements.
Then, for a desired decay rate $\mu>0$, choose $M$ and $\lambda>0$ according to Theorem~\ref{thm:main-stabilization}.
(Practically, $\lambda$ can be tuned numerically by monitoring the decay of $\|y(t)-\bar y\|_{V'}$.)

\paragraph{Online implementation by sampled measurements.}
The law \eqref{eq:alg-u-ideal} is the continuous-time ideal feedback used in the analysis.  In the physical implementation we choose update times $0=t_0<t_1<\cdots<t_N=T$, with $\tau_n:=t_{n+1}-t_n$ small, and use the measured physical temperature at $t_n$ to update the command on the next interval.  At an update time $t_n$:
\begin{enumerate}
\item[(i)] Form the error state $z_n=y_{\rm meas}(t_n)-\bar y$ from the physical temperature measurement.
\item[(ii)] Solve the elliptic problem for $w_n=\mathcal A^{-1}z_n$:
\[
w_n - \kappa_m\Delta w_n = z_n\ \text{in }\Omega,\qquad \partial_\nu w_n=0\ \text{on }\partial\Omega.
\]
\item[(iii)] Set the updated command $u_{\mathrm{ideal},j}^n=(u_r)_j-\lambda\,w_n(z_j)$ for $j=1,\dots,M$.
\end{enumerate}

\textit{Remark on sensors.}
The above requires sufficiently informative measurements of the physical state at $t_n$ in order to form $z_n$ and solve the elliptic problem.
If only boundary (or locally averaged) measurements are available, one needs a state estimation step; this lies beyond
the present scope and can be incorporated by standard observer constructions once the model is fixed.

\paragraph{Short update windows and thermal memory.}
The discrete implementation does not assert that the heat equation loses memory at the sampling times.  The measured state at $t_n$ contains all heat generated in the preceding steps and is used as the nonzero initial datum of the next physical forward problem.  What the update accomplishes is to recompute the control from this actual state, rather than propagating an approximate model open-loop from time $0$ to the final treatment time.  Consequently, the whole-space/bounded-domain comparison is invoked only on the next short prediction window, with the accumulated thermal history entering through the measured initial state in Lemma~\ref{lem:R3-to-Omega-restriction}. \color{black}{The corresponding cumulative restriction error is given by \eqref{eq:cumulative-domain-error}.}

\subsection{Identification of the thermo--plasmonic realization map}
\label{subsec:alg-Ktp}

\noindent
The thermo--plasmonic analysis of Section~\ref{sec:coupling} gives, on the complete update
interval,
\begin{equation}\label{eq:alg-G-Ktp}
 G=\mathcal K_\delta p+\mathcal R^\delta p.
\end{equation}
Here $p\in U$ is an illumination history and $G\in Y$ is the corresponding vector of effective
point-source histories.  Because $\mathcal K_\delta$ contains the Volterra solution operator,
\eqref{eq:alg-G-Ktp} is not a pointwise-in-time matrix relation.

The finite-dimensional map used in the algorithm is calibrated as follows.
\begin{enumerate}
\item[(i)] Choose linearly independent causal profiles
$\vartheta_1,\ldots,\vartheta_Q$ and the illumination dictionary
\[
 \psi_{\ell q}:=e_\ell\vartheta_q,
 \qquad 1\le\ell\le P,\quad 1\le q\le Q,
\]
so that $U_0=\operatorname{span}\{\psi_{\ell q}\}$ as in
\eqref{eq:U0-temporal-dictionary}.

\item[(ii)] For every basis history $\psi_{\ell q}$, compute or measure the effective output history
\[
 g_{\ell q}:=\mathcal K_\delta\psi_{\ell q}\in Y.
\]
Define
\[
 Y_0:=\operatorname{span}\{g_{\ell q}:1\le\ell\le P,\ 1\le q\le Q\}
      =\mathcal K_\delta(U_0).
\]
Thus the Volterra-distorted temporal shapes are included in the calibrated output basis rather than
being prescribed in advance.

\item[(iii)] Orthonormalize a maximal independent subset of the histories $g_{\ell q}$ in the
$Y$ inner product.  In these coordinates assemble the matrix
\[
 \mathbf K_{\delta,Q}\in\mathbb R^{MQ\times PQ}
\]
representing $\mathcal K_\delta|_{U_0}:U_0\to Y_0$.  Under
Assumption~\ref{ass:illumination}, Proposition~\ref{prop:K-right-inverse} gives
$\operatorname{rank}\mathbf K_{\delta,Q}=MQ$.

\item[(iv)] Compute the SVD of $\mathbf K_{\delta,Q}$.  Its smallest nonzero singular value is the
calibrated actuation constant $c_{\rm act}(\delta,Q)$ in
\eqref{eq:actuation-singular-value}; it determines the inverse-amplification factor
$C_{\rm inv}(\delta,Q)$.  If this singular value is too small, one should change or enrich the
illumination/temporal dictionary before applying the controller.

\item[(v)] Form the minimum-norm history right-inverse
$\mathcal K_{\delta,Q}^{\dagger}:Y_0\to U_0$ from
\eqref{eq:operator-right-inverse}.  For $u_{\rm des}\in Y_0$ the implemented history is
\[
 p=\mathcal K_{\delta,Q}^{\dagger}u_{\rm des},
 \qquad
 \mathcal K_\delta p=u_{\rm des}.
\]
The physical heat-input discrepancy is then solely the asymptotic remainder
$\mathcal R^\delta p$, with
\[
 \|\mathcal R^\delta p\|_Y
 \le C_T M\delta^\mu C_{\rm inv}(\delta,Q)\|u_{\rm des}\|_Y.
\]
\end{enumerate}

\subsection{Control loop with projection onto $Y_0$}
\label{subsec:alg-loop}

\noindent
In general, the ideal control signal $u_{\mathrm{ideal}}$ generated by \eqref{eq:alg-u-ideal} does not belong to $Y_0$. We therefore project $u_{\mathrm{ideal}}$ onto $Y_0$ as in \eqref{eq:5.Y0-u-des}:
\begin{equation}
\label{eq:alg-u-des}
u_{\mathrm{des}} := P_{Y_0} u_{\mathrm{ideal}} \in Y_0.
\end{equation}
The auxiliary trajectory $y_{Y_0}$ driven by $u_{\mathrm{des}}$ is defined by \eqref{eq:5.Y0}.
Proposition~\ref{prop:projection-error} quantifies the deviation between $y_{Y_0}$ and the ideal trajectory driven by $u_{\mathrm{ideal}}$ in the $V'$ framework (and in $H$ after any $t_0>0$).
\noindent
To realize $u_{\mathrm{des}}$ physically, we choose the complete illumination history
$p\in U_0$ by
\begin{equation}
\label{eq:alg-p-Kdagger}
p:=\mathcal K_{\delta,Q}^{\dagger}u_{\mathrm{des}}.
\end{equation}
Then $\mathcal K_\delta p=u_{\mathrm{des}}$, while the physical discrepancy is the controlled
remainder $\mathcal R^\delta p$ in \eqref{eq:alg-G-Ktp}.  Theorem~\ref{thm:main-nanoparticle-tracking}
therefore yields bounds on $y_{\mathrm{phys}}-y_{Y_0}$ in $V'$ (and in $H$ after $t_0>0$).
Combining both contributions gives Corollary~\ref{cor:total-error}.

For the sampled implementation on $I_n=[t_n,t_{n+1}]$, we use the same construction after
translation and rescaling of the temporal profiles.  We denote the corresponding spaces and
operators by
\[
 U_{0,n},\qquad Y_{0,n}:=\mathcal K_{\delta,n}(U_{0,n}),\qquad
 \mathcal K_{\delta,Q,n}^{\dagger}:Y_{0,n}\to U_{0,n}.
\]

\medskip

\noindent\textbf{Summary of the implementable loop.}
\begin{enumerate}
\item[(i)] \textbf{Choose actuator points and gain.}
Select $\textcolor{black}{z_1},\dots,\textcolor{black}{z_M}$ so that Assumption~\ref{assump:KRW-act} holds and choose $\lambda>0$
according to Theorem~\ref{thm:main-stabilization} for the desired decay rate in $V'$.

\item[(ii)] \textbf{(If needed) fixed-point pre-compensation on $X_N$.}
If $y_r\in X_N$ is not an equilibrium of the free Neumann heat dynamics, solve \eqref{eq:fixed-point} for
$y^\star$ and set $\bar y:=y^\star$; otherwise set $\bar y:=y_r$.
Compute the corresponding feedforward $u_r$ by solving \eqref{eq:feedforward-linear-system}
(equivalently $u_r=U_N\bar y$).

\item[(iii)] \textbf{Calibrate the reachable finite-dimensional map.}
Choose the temporal/illumination dictionary defining $U_0$, generate the complete output histories
$g_{\ell q}=\mathcal K_\delta(e_\ell\vartheta_q)$, set $Y_0=\mathcal K_\delta(U_0)$, and compute
$\mathbf K_{\delta,Q}$, $c_{\rm act}(\delta,Q)$, and
$\mathcal K_{\delta,Q}^{\dagger}$ as in Subsection~\ref{subsec:alg-Ktp}.

\item[(iv)] \textbf{At each update time $t_n$ (measurement + projection + EM actuation).}
\begin{enumerate}
\item Measure the physical temperature produced by the original heterogeneous nanoparticle model and form $z_n=y_{\rm meas}(t_n)-\bar y$.
\item Solve $w_n=\mathcal A^{-1}z_n$, i.e.\ $w_n-\kappa_m\Delta w_n=z_n$ in $\Omega$ with $\partial_\nu w_n=0$ on $\partial\Omega$.
\item Set the vector command
$u_{\mathrm{ideal},j}^n=(u_r)_j-\lambda w_n(z_j)$.  Choose a causal hold/pulse profile
$\chi_n$ on $I_n=[t_n,t_{n+1}]$ and form the desired command history
\[
 \widehat u_{\mathrm{ideal}}^{\,n}(t):=\chi_n(t)u_{\mathrm{ideal}}^n,
 \qquad t\in I_n.
\]
Using the translated/scaled temporal dictionary on $I_n$, set
$u_{\mathrm{des}}^n=P_{Y_{0,n}}\widehat u_{\mathrm{ideal}}^{\,n}$.
\item Compute the illumination history
$p_n=\mathcal K_{\delta,Q,n}^{\dagger}u_{\mathrm{des}}^n\in U_{0,n}$, apply this history on
$I_n$ (possibly supported only on a shorter subinterval), and then measure again.
\item For the comparison used in the analysis, the physical forward problem on $[t_n,t_{n+1}]$ starts from the nonzero measured state at $t_n$.  Proposition~\ref{prop:nonzero-initial-reduction} compares this original heterogeneous evolution with the effective whole-space point-source model, and Lemma~\ref{lem:R3-to-Omega-restriction} compares the latter with the bounded-domain model on $B_{\rm obs}$ over the short window $\tau_n$.
\end{enumerate}
\end{enumerate}

\section{Conclusions and perspectives} \label{8}

\noindent 
We have combined a discrete effective-medium description of heat generation by
plasmonic nanoparticles with an abstract feedback design for heat-profile
tracking based on point actuators. At the modeling level, we relied on the
rigorous point-approximation framework of \cite{CaoMukherjeeSini2025} to
identify a finite collection of nanoparticles as a system of effective point
actuators for the heat equation: the temperature variation induced by the
particles can be represented, up to small error terms, as a superposition of
heat kernels centered at the nanoparticle locations, with amplitudes solving
a Volterra system driven by the internal electromagnetic energy.

\medskip
\noindent 
At the control level, we interpret the heat equation with point actuators in the natural
space $V'$, \color{black}{$V:=H^1(\Omega)$, with the Dirac forcing acting in $D(\mathcal A)'$,} and consider a feedback law based on the
elliptically smoothed measurements $\mathcal A^{-1}y$.
Using the stabilization theory of \cite{KunRodWal24-cocv}, we obtain exponential decay of the
tracking error in $V'$ under an actuator-family hypothesis, and we recover $H$--tracking after any
positive time $t_0>0$ via parabolic regularization.

\medskip
\noindent 
Finally, we coupled the two levels by expressing the effective heat inputs
generated by the nanoparticle ensemble as a linear history operator $\mathcal K_\delta$
acting on electromagnetic illumination histories.  A finite illumination dictionary defines
$U_0$, while the realizable heat-input space is the actual reachable range
$Y_0=\mathcal K_\delta(U_0)$, which automatically incorporates the Volterra temporal dynamics.
Under the non-degenerate illumination assumption, the restricted map has full rank and admits a
minimum-norm right-inverse on $Y_0$.  The realization estimate retains the corresponding inverse-conditioning factor, so the discrepancy between the ideal feedback inputs and the effective
thermo--plasmonic ones is quantified in terms of the small geometric parameter $\delta$, the
modeling error from the discrete effective description, and the calibrated smallest singular value. This yields an abstract estimate
on the tracking error between the ideal closed-loop temperature and the
physically realized one, and shows that, in the asymptotic regime considered
in \cite{CaoMukherjeeSini2025}, plasmonic nanoparticles can indeed be used as
actuators for heat-profile tracking.

\medskip
\noindent 
We mention two idealizations that are used in the feedback design. First, the coefficient vector $p(t):=(p_l(t))^P_{l=1}\in\mathbb{R}^P$ is treated as a signed control acting linearly through the illumination dictionary (coming from the electromagnetic fields generated by the injected nanoparticles), but physically it appears of the form $p_\ell(t)=|a_\ell(t)|^2\ge 0$ (where $a_\ell$ are related to the electric field intensity). This can be handled by restricting to effective inputs attainable with $p \ge 0$. Second, the output map is taken as the elliptically regularized point evaluation $Cz=\big((\mathcal{A}^{-1}z)(z_j)\big)_{j=1}^M$. \textcolor{black}{Along the parabolic trajectories one has $z(t)\in H=L^2(\Omega)$ for a.e. $t$, and $\mathcal A^{-1}:H\to D(\mathcal A)\hookrightarrow C(\overline\Omega)$, so these pointwise traces are well-defined.} However, in practice, this corresponds to one more step of reconstructing these regularized outputs from physically available measurements. These two issues will be discussed in a forthcoming work, alongside numerical studies illustrating the closed-loop tracking performance and an analysis working directly on the full space to avoid finite-dimensional projections.
\bigskip

\noindent Several directions for further work remain open, including heterogeneous and
structured backgrounds and lower-dimensional configurations with
 robust and optimal designs of nanoparticle configurations
and illuminations, and numerical validation. In addition, we plan to analyze the feedback control part in the whole space, with the diffusion operator $A$ stated now in the whole space with absolutely continuous spectrum. In the same direction, we would like to study tracking other fields, as acoustic fluctuation (in the framework of acoustic cavitation) using either purely acoustic models, see \cite{Mukherjee2024}, in Ultrasound, or \textcolor{black}{multi-physics} models as photo-acoustics, see \cite{Ghandriche2022}. Finally, such an approach makes also full sense in mechanics where the goal is to generate and track full displacements and not only pressure parts, as in acoustics.  
\bigskip


\section{Appendix}\label{ap}

\subsection{Spectral assumptions}
\noindent
Proposition~\ref{prop:generic-rank-CM} shows that, for fixed $M$ and domain
$\Omega$, Assumption~\ref{ass:rank-4} is a generic property of the actuator
positions $(\textcolor{black}{z_1},\dots,\textcolor{black}{z_M})$: it fails only on an exceptional set
$\mathcal{Z}_M$ of Lebesgue measure zero. Equivalently, if the points
$z_j$ are chosen ``at random'' in $\Omega$ (according to any absolutely
continuous probability distribution), then the matrix
$\mathcal{B}_{jk}=\phi_k(z_j)$ is invertible with probability one.
\newline
In particular, the rank condition can always be enforced in practice by
perturbing a given configuration of points slightly (as long as the points
remain in the interior of $\Omega$ and do not approach the boundary or
coincide). In simple geometries, such as intervals or rectangles where the
eigenfunctions are explicitly known (e.g.\ trigonometric functions), it is
enough to ensure that:
\begin{enumerate}
\item the points $z_j$ do not lie on nodal sets of any of the eigenfunctions
$\phi_k$ with $k\le M$, and
\item the points are not arranged in a way that respects the symmetries of
these nodal sets.
\end{enumerate}
For example, on an interval $\Omega=(0,L)$ with Neumann boundary conditions,
the eigenfunctions are $\phi_1\equiv \mathrm{const}$ and
$\phi_k(x)=\cos((k-1)\pi x/L)$ for $k\ge2$. Choosing $M$ distinct points
$z_j\in(0,L)$ that avoid the zeros of these cosine functions and are not
symmetrically paired with respect to the midpoint generically yields an
invertible matrix $\mathcal{B}$.

\subsubsection{A concrete condition in 3D for a Neumann cube}
\noindent
We now give an explicit sufficient condition for a three-dimensional cube,
using tensor products of the 1D construction.

\begin{lemma}[Invertibility on a 3D Neumann cube via tensor-product DCT grid]
\label{lem:rank-3D}
Let $\Omega=(0,L_1)\times(0,L_2)\times(0,L_3)$ and let
$A=\kappa_m\Delta$ with Neumann boundary conditions on $\partial\Omega$.
An orthonormal basis of eigenfunctions of $-A_0$ is given (up to
normalization) by
\[
\phi_{\mathbf{n}}(x)
= c_{\mathbf{n}}\,
\cos\!\Big(\frac{n_1\pi x_1}{L_1}\Big)
\cos\!\Big(\frac{n_2\pi x_2}{L_2}\Big)
\cos\!\Big(\frac{n_3\pi x_3}{L_3}\Big),
\]
where $\mathbf{n}=(n_1,n_2,n_3)\in\mathbb{N}_0^3$ and $c_{\mathbf{n}}>0$ are
normalizing constants.
\newline
Fix integers $N_1,N_2,N_3\ge0$ and set
\[
I := \{\,\mathbf{n}=(n_1,n_2,n_3)\in\mathbb{N}_0^3 : 0\le n_\ell\le N_\ell,
\ \ell=1,2,3 \,\}.
\]
Let
\[
H_M := \mathrm{span}\{\phi_{\mathbf{n}} : \mathbf{n}\in I\},
\quad M:=(N_1+1)(N_2+1)(N_3+1).
\]
Index $H_M$ by an enumeration $\{\phi_k\}_{k=1}^M$ of $\{\phi_{\mathbf{n}} :
\mathbf{n}\in I\}$ (the order is irrelevant for rank). For each direction
$\ell=1,2,3$ define
\[
\textcolor{black}{z}^{(\ell)}_{j_\ell} := \frac{(2j_\ell-1)L_\ell}{2(N_\ell+1)},
\qquad j_\ell=1,\dots,N_\ell+1,
\]
and form the tensor-product grid
\[
\textcolor{black}{z}_{\mathbf{j}}
:= \big(\textcolor{black}{z}^{(1)}_{j_1},\textcolor{black}{z}^{(2)}_{j_2},\textcolor{black}{z}^{(3)}_{j_3}\big),
\qquad
\mathbf{j}=(j_1,j_2,j_3)\in J,
\]
where $J:=\{1,\dots,N_1+1\}\times\{1,\dots,N_2+1\}\times\{1,\dots,N_3+1\}$
and $|J|=M$.
\newline
Let $\mathcal{B}\in\mathbb{R}^{M\times M}$ be the matrix
\[
\mathcal{B}_{\mathbf{j},\mathbf{n}} := \phi_{\mathbf{n}}(\textcolor{black}{z}_{\mathbf{j}}),
\qquad \mathbf{j}\in J,\ \mathbf{n}\in I,
\]
and identify $\mathcal{B}$ with an $M\times M$ matrix after any fixed
enumeration of $J$ and $I$. Then $\mathcal{B}$ is invertible.
\end{lemma}

\begin{proof}
By the tensor-product structure of the eigenfunctions and the grid, we have
\[
\phi_{\mathbf{n}}(\textcolor{black}{z}_{\mathbf{j}})
= c_{\mathbf{n}}\,
\prod_{\ell=1}^3 \cos\!\Big(\frac{n_\ell\pi \textcolor{black}{z}^{(\ell)}_{j_\ell}}{L_\ell}\Big).
\]
Define, for each $\ell=1,2,3$, the 1D matrices
\[
C^{(\ell)}_{j_\ell,n_\ell}
:= \cos\!\Big(\frac{n_\ell\pi \textcolor{black}{z}^{(\ell)}_{j_\ell}}{L_\ell}\Big),
\qquad j_\ell=1,\dots,N_\ell+1,\quad n_\ell=0,\dots,N_\ell.
\]
By replacing $M$ by
$N_\ell+1$, each $C^{(\ell)}$ is, up to normalization, a DCT-II matrix and
hence invertible.
\newline
Up to multiplication by nonzero diagonal matrices corresponding to the
normalization constants $c_{\mathbf{n}}$, the full matrix $\mathcal{B}$ can
be written as a (permuted) Kronecker product
\[
\mathcal{B} = D \cdot\big(C^{(1)}\otimes C^{(2)}\otimes C^{(3)}\big),
\]
where $D$ is a diagonal matrix with strictly positive entries and
$\otimes$ denotes the Kronecker product. Since each $C^{(\ell)}$ is
invertible and the Kronecker product of invertible matrices is invertible,
$C^{(1)}\otimes C^{(2)}\otimes C^{(3)}$ is invertible. Multiplication by
the diagonal matrix $D$ does not affect invertibility. Therefore
$\mathcal{B}$ is invertible. In particular, any selection of $N\le M$ rows yields an $N\times M$ matrix with full row rank~$N$.
\end{proof}

\subsection[Verification of Assumption~\ref{assump:KRW-act} on $C^{1,1}$ domains]{Verification of Assumption~\ref{assump:KRW-act} on $C^{1,1}$ domains}
\label{subsec:app-KRW-act-C11}

\noindent
This subsection provides a self-contained verification of the actuator hypothesis
(Assumption~\ref{assump:KRW-act}) for Dirac actuators on smooth domains.
The only part of Assumption~\ref{assump:KRW-act} that is geometric is the divergence
$\xi_M^+\to+\infty$ of the nullspace coercivity constant.
In \cite{KunRodWal24-cocv} the authors verify this in rectangles/cubes
by a self-similar decomposition argument. Here we replace that step by a standard
finite-element interpolation estimate on shape-regular quasi-uniform meshes, which works on general
$C^{1,1}$ domains (hence extending the class of box/polyhedral domains  in \cite{KunRodWal24-cocv} to a class of domains admitting a suitable meshing property; see~\eqref{eq:app-shapereg}).

\medskip

\noindent
Throughout, $\Omega\subset\R^d$ is bounded of class $C^{1,1}$ with $d\in\{1,2,3\}$, $A_0=\kappa_m\Delta$ with
homogeneous Neumann boundary condition, and $\mathcal A=I-A_0=I-\kappa_m\Delta$ as in the main text.
We recall that $D(\mathcal A)=\{w\in H^2(\Omega):\partial_\nu w=0\text{ on }\partial\Omega\}$ and
$V=H^1(\Omega)$.

\begin{remark}[Actuator points on $\overline\Omega$] 
\label{rem:app-KRW-closure}
In the main text we fix actuator locations $z_j\in\Omega$ for simplicity. The functional-analytic framework used
throughout only requires $\delta_x\in D(\mathcal A)'$, which holds for every $x\in\overline\Omega$ because
$D(\mathcal A)\hookrightarrow C(\overline\Omega)$ (see Proposition~\ref{prop:B-bounded-DA-dual}).
Therefore, in the verification below we may select actuator points among the vertices of a triangulation, which may
may include boundary vertices at the level of the mesh construction.  For point actuators in the physical model, such boundary nodes can be replaced by neighboring interior points at distance $\mathcal O(h)$ from $\partial\Omega$; shape regularity and the fill-distance estimate are preserved up to uniform constants, so this does not affect Assumption~\ref{assump:KRW-act} nor the stabilization statement.
\end{remark}

\begin{lemma}[Elliptic graph norm equivalence]
\label{lem:app-graphnorm-H2}
There exist constants $c_\mathcal{A},C_\mathcal{A}>0$ depending only on $\Omega$ and $\kappa_m$ such that
\begin{equation}
\label{eq:app-graphnorm-H2}
c_\mathcal{A}\,\|w\|_{H^2(\Omega)} \le \|w\|_{D(\mathcal A)} \le C_\mathcal{A}\,\|w\|_{H^2(\Omega)}
\qquad \forall w\in D(\mathcal A).
\end{equation}
In particular, point evaluation $w\mapsto w(x)$ is continuous on $D(\mathcal A)$ for every $x\in\overline\Omega$.
\end{lemma}

\begin{proof}
The upper bound in \eqref{eq:app-graphnorm-H2} follows from $\mathcal A w=w-\kappa_m\Delta w$ and the estimate
$\|w\|_{L^2}+\|\Delta w\|_{L^2}\lesssim\|w\|_{H^2}$. For the reverse inequality, let $f:=\mathcal A w\in L^2(\Omega)$.
Then $w$ solves the Neumann problem $(I-\kappa_m\Delta)w=f$.
On bounded $C^{1,1}$ domains, $H^2$-regularity holds for this uniformly elliptic operator with Neumann boundary condition (see, e.g., \cite{GilbargTrudinger2001}),
so $\|w\|_{H^2}\le C\|f\|_{L^2}=C\|w\|_{D(\mathcal A)}$.
The continuity of point evaluation follows from the embedding $H^2(\Omega)\hookrightarrow C(\overline\Omega)$ for $d\le3$.
\end{proof}

\medskip
\noindent\textbf{Choice of actuator families.}
Let $\{\mathcal T_h\}_{h\downarrow 0}$ be a family of shape-regular, quasi-uniform simplicial triangulations of $\Omega$.
Denote by $\mathcal N_h=\{\textcolor{black}{z}_{h,1},\dots,\textcolor{black}{z}_{h,M(h)}\}\subset\overline\Omega$ the set of all mesh vertices.
For each $h$ (equivalently, for each $M:=M(h)$), define Dirac actuators
\begin{equation}
\label{eq:app-actuators-dirac}
d_{M,j}:=\delta_{\textcolor{black}{z}_{h,j}}\in D(\mathcal A)',\qquad j=1,\dots,M.
\end{equation}
To produce biorthogonal test functions in $D(\mathcal A)$, we use Riesz representatives.

\begin{lemma}[Biorthogonal family in $D(\mathcal A)$]
\label{lem:app-biorthogonal}
For each $M$ and each $j\in\{1,\dots,M\}$ there exists $\Psi_{M,j}\in D(\mathcal A)$ such that
\begin{equation}
\label{eq:app-biorthogonality}
\langle d_{M,i},\Psi_{M,j}\rangle = \Psi_{M,j}(\textcolor{black}{z}_{h,i}) = \delta_{ij},\qquad i,j=1,\dots,M.
\end{equation}
Consequently, $\dim\mathrm{span}\{d_{M,j}\}_{j=1}^M=\dim\mathrm{span}\{\Psi_{M,j}\}_{j=1}^M=M$.
\end{lemma}

\begin{proof}
By Lemma~\ref{lem:app-graphnorm-H2}, each evaluation functional $d_{M,j}$ is continuous on the Hilbert space
$D(\mathcal A)$ equipped with the inner product $(u,v)_{D(\mathcal A)}:=(\mathcal A u,\mathcal A v)_H$.
Hence, by the Riesz representation theorem, there exists a unique $r_{M,j}\in D(\mathcal A)$ such that
$\langle d_{M,j},w\rangle=(w,r_{M,j})_{D(\mathcal A)}$ for all $w\in D(\mathcal A)$.
Define the Gram matrix $G_M\in\R^{M\times M}$ by
\[
(G_M)_{ij}:=\langle d_{M,i},r_{M,j}\rangle = (r_{M,i},r_{M,j})_{D(\mathcal A)}.
\]
For any $c\in\R^M$, we have $c^\top G_M c=\big\|\sum_{j=1}^M c_j r_{M,j}\big\|_{D(\mathcal A)}^2$, so $G_M$ is symmetric
positive definite and thus invertible. Now set
\[
\Psi_{M,j} := \sum_{\ell=1}^M (G_M^{-1})_{j\ell}\, r_{M,\ell}\in D(\mathcal A).
\]
Then, for every $i$,
\[
\langle d_{M,i},\Psi_{M,j}\rangle
=\sum_{\ell=1}^M (G_M^{-1})_{j\ell}\,\langle d_{M,i},r_{M,\ell}\rangle
=\sum_{\ell=1}^M (G_M^{-1})_{j\ell}\,(G_M)_{i\ell} = \delta_{ij},
\]
which is \eqref{eq:app-biorthogonality}. The dimension claims follow immediately from biorthogonality.
\end{proof}

\medskip
\noindent\textbf{Divergence of the nullspace coercivity constant.}
Let $V_h\subset H^1(\Omega)$ be the standard piecewise-affine ($P_1$) finite element space associated with $\mathcal T_h$,
and let $I_h:H^2(\Omega)\to V_h$ be the nodal Lagrange interpolant.

\begin{lemma}[Nodal interpolation estimate for polyhedral domains]
\label{lem:app-interp}
There exists a constant $C_I>0$, independent of $h$, such that for all $w\in H^2(\Omega)$,
\begin{equation}
\label{eq:app-interp}
\|w-I_h w\|_{H^1(\Omega)} \le C_I\, h\,\|w\|_{H^2(\Omega)}.
\end{equation}
\end{lemma}

\begin{proof}
This is the classical $P_1$ interpolation estimate on shape-regular quasi-uniform meshes; see Theorem $4.4.20$ in \cite{BrennerScottFEM} for polyhedral domains.
\end{proof}

\medskip
 \noindent To deal with more general domains as polyhedral ones, we take $\Omega\subset\R^3$ to be a bounded $C^{1,1}$ domain and work with a boundary-fitted isoparametric mesh family (see, e.g., \cite{AzouaniTiti14, CiarletRaviart1972,EggerFritzKunRod25, Lenoir1986, LunasinTiti17})
$\{\mathcal T_h\}_{h\downarrow 0}$: each element $K\in\mathcal T_h$ is the image of a fixed reference simplex
$\widehat K$ under a bijective map $F_K\in W^{2,\infty}(\widehat K;\R^3)$.
We assume uniform (isoparametric) shape-regularity and quasi-uniformity: there exist constants $\gamma,c_0>0$,
independent of $h$ and $K$, such that
\begin{equation}\label{eq:app-shapereg}
\|DF_K\|_{L^\infty(\widehat K)}\le \gamma h_K,\qquad
\|(DF_K)^{-1}\|_{L^\infty(\widehat K)}\le \gamma h_K^{-1},\qquad
\|D^2F_K\|_{L^\infty(\widehat K)}\le \gamma h_K,
\end{equation}
and $\det DF_K \ge c_0\, h_K^{3}$ a.e. on $\widehat K$.
Moreover, $h\simeq h_K$ for all $K\in\mathcal T_h$, and $\Omega=\bigcup_{K\in\mathcal T_h}K$.

\medskip
\noindent\emph{Parametrizable domains.}
The framework above also applies if $\Omega$ is given as the image of a reference domain.
More precisely, assume there exist a bounded reference domain $\widehat\Omega\subset\R^3$ admitting a
shape-regular simplicial mesh family $\{\widehat{\mathcal T}_h\}_{h\downarrow 0}$ with
$\widehat\Omega=\bigcup_{\widehat K\in\widehat{\mathcal T}_h}\widehat K$, and a bi-Lipschitz map
$\Phi:\widehat\Omega\to\Omega$ with $\Phi\in W^{2,\infty}(\widehat\Omega;\R^3)$ and $\Phi^{-1}\in W^{1,\infty}(\Omega;\R^3)$.
Then, pushing forward the reference mesh by $\Phi$ yields an isoparametric family (cf.\ the standard pushforward construction in \cite{CiarletRaviart1972,Zlamal1973})
$\mathcal T_h:=\{\Phi(\widehat K):\widehat K\in\widehat{\mathcal T}_h\}$ satisfying the above
uniform bounds (possibly with a different constant $\gamma$) and, in particular, $\Omega=\bigcup_{K\in\mathcal T_h}K$. This shows how rich is the family of domains we can deal with.

The (isoparametric) $P_1$ finite element space is defined by pullback: for each $K$, functions in $V_h$
restrict to polynomials of degree $\le 1$ on $\widehat K$ after composition with $F_K^{-1}$.
The operator $I_h:H^2(\Omega)\to V_h$ denotes the nodal Lagrange interpolant (defined using the mesh vertices).

\begin{lemma}[Nodal interpolation estimate for more general domains]
\label{lem:app-interp-curved}
There exists a constant $C_I>0$, independent of $h$, such that for all $w\in H^2(\Omega)$,
\begin{equation}
\label{eq:app-interp-curved}
\|w-I_h w\|_{H^1(\Omega)} \le C_I\, h\,\|w\|_{H^2(\Omega)}.
\end{equation}
\end{lemma}

\begin{proof}
We give a self-contained proof that remains valid for $C^{1,1}$ domains under the mesh assumption above.
Fix $w\in H^2(\Omega)$ and $K\in\mathcal T_h$.
Set $\widehat w:=w\circ F_K\in H^2(\widehat K)$ and denote by $\widehat I$ the nodal $P_1$ interpolant on the reference simplex
$\widehat K$ (with the reference vertices as nodes).
By construction of the isoparametric interpolant, the local interpolant satisfies
\begin{equation}
\label{eq:app-local-interp}
(I_hw)|_{K} = (\widehat I\,\widehat w)\circ F_K^{-1}.
\end{equation}

\smallskip
\noindent\emph{Step 1: reference-element estimate.}
On the fixed simplex $\widehat K$ the Bramble--Hilbert lemma (or an explicit finite-dimensional argument) yields the standard bound
\begin{equation}
\label{eq:app-ref-est}
\|\widehat w-\widehat I\widehat w\|_{H^1(\widehat K)}\le C_{\widehat I}\,|\widehat w|_{H^2(\widehat K)},
\end{equation}
where $C_{\widehat I}$ depends only on the reference simplex.

\smallskip
\noindent\emph{Step 2: mapping of norms.}
Using the change of variables $x=F_K(\widehat x)$, $J_K:=DF_K$, and the chain rule,
\[\nabla_x \big(\phi\circ F_K^{-1}\big)(x)=J_K(\widehat x)^{-\top}\,\nabla_{\widehat x}\phi(\widehat x),\qquad \widehat x=F_K^{-1}(x),\]
we obtain the local $H^1$-estimate
\begin{equation}
\label{eq:app-H1-map}
\|w-I_hw\|_{H^1(K)}\le C\,h_K^{\frac d2-1}\,\|\widehat w-\widehat I\widehat w\|_{H^1(\widehat K)},
\end{equation}
with a constant $C$ depending only on the uniform bounds in \eqref{eq:app-shapereg}.
Similarly, by differentiating twice and using \eqref{eq:app-shapereg}, one has the standard pullback bound
\begin{equation}
\label{eq:app-H2-map}
|\widehat w|_{H^2(\widehat K)}\le C\,h_K^{2-\frac d2}\,\|w\|_{H^2(K)},
\end{equation}
where again $C$ depends only on \eqref{eq:app-shapereg}.
(A direct derivation expands $D^2_{\widehat x}(w\circ F_K)$ into a term involving $D^2_x w$ and a lower-order term involving
$D_x w$ multiplied by $D^2F_K$; the latter is controlled by \eqref{eq:app-shapereg} and absorbed into $\|w\|_{H^2(K)}$.)

\smallskip
\noindent\emph{Step 3: local-to-global.}
Combining \eqref{eq:app-ref-est}--\eqref{eq:app-H2-map} gives
\[
\|w-I_hw\|_{H^1(K)}
\le C\,h_K^{\frac d2-1}\,C_{\widehat I}\,|\widehat w|_{H^2(\widehat K)}
\le C\,C_{\widehat I}\,h_K\,\|w\|_{H^2(K)}.
\]
Summing over $K\in\mathcal T_h$ and using quasi-uniformity ($h_K\simeq h$) yields
\[
\|w-I_hw\|_{H^1(\Omega)}^2
=\sum_{K\in\mathcal T_h}\|w-I_hw\|_{H^1(K)}^2
\le C\,h^2\sum_{K\in\mathcal T_h}\|w\|_{H^2(K)}^2
=C\,h^2\,\|w\|_{H^2(\Omega)}^2,
\]
which implies \eqref{eq:app-interp-curved} with $C_I:=\sqrt{C}$.
The constant depends only on the reference simplex and the uniform $C^1$ shape-regularity bounds \eqref{eq:app-shapereg},
and is therefore independent of $h$.
\end{proof}

\begin{proposition}[Nullspace coercivity on $C^{1,1}$ domains]
\label{prop:app-xi-diverges}
Let $D_M\subset D(\mathcal A)$ be the subspace
\[
\begin{aligned}
D_M:=\bigl\{w\in D(\mathcal A):\;&\langle d_{M,j},w\rangle=w(\textcolor{black}{z}_{h,j})=0\\
&\text{for every }j=1,\dots,M\bigr\}.
\end{aligned}
\]
Then there exists $c>0$, independent of $h$, such that the constant
\[
\xi_M^+ := \inf\Big\{\frac{\|w\|_{D(\mathcal A)}^2}{\textcolor{black}{\|w\|_{H^1(\Omega)}^2}}:\ w\in \textcolor{black}{D_M}\setminus\{0\}\Big\}
\]
satisfies
\begin{equation}
\label{eq:app-xi-rate}
\xi_M^+ \ge c\,h^{-2}\longrightarrow +\infty\qquad (h\downarrow 0).
\end{equation}
\end{proposition}

\begin{proof}
Let $w\in D_M$. Since $w\in D(\mathcal A)\subset H^2(\Omega)$ and $w(\textcolor{black}{z}_{h,j})=0$ for every vertex, the nodal interpolant
satisfies $I_h w\equiv 0$.
Therefore, using Lemma~\ref{lem:app-interp-curved} and Lemma~\ref{lem:app-graphnorm-H2}, we obtain
\[
\|w\|_{H^1(\Omega)} = \|w-I_h w\|_{H^1(\Omega)} \le C_I h\,\|w\|_{H^2(\Omega)}
\le  C \, h\,\|w\|_{D(\mathcal A)},
\]
with $C:=\frac{C_I}{c_\mathcal A}$ independent
of $h$. Squaring and taking reciprocals yields
\[
\frac{\|w\|_{D(\mathcal A)}^2}{\textcolor{black}{\|w\|_{H^1(\Omega)}^2}} \ge C^{-2} h^{-2}.
\]
Taking the infimum over $w\in D_M\setminus\{0\}$ gives \eqref{eq:app-xi-rate}.
\end{proof}

\begin{corollary}[Assumption~\ref{assump:KRW-act} on domains satisfying (\ref{eq:app-shapereg})]
\label{cor:app-KRW-act}
With the choices \eqref{eq:app-actuators-dirac} and the biorthogonal family from Lemma~\ref{lem:app-biorthogonal},
Assumption~\ref{assump:KRW-act} holds on bounded $C^{1,1}$ domains, and in particular
$\xi_M^+\to+\infty$.
\end{corollary}

\begin{remark}[Rate in terms of the number of actuators]
If $\{\mathcal T_h\}$ is quasi-uniform, then $M=M(h)\asymp h^{-d}$, hence \eqref{eq:app-xi-rate} implies
$\xi_M^+\gtrsim M^{2/d}$.
\end{remark}

\subsection{Three sufficient mechanisms ensuring solvability of the fixed--point correction}
\label{subsec:fp-solvability-mechanisms}

\noindent
This subsection collects three \emph{complementary} ways to guarantee solvability of the finite-dimensional
fixed--point correction used to compensate the steady bias. The three mechanisms correspond to:
(i) increasing the medium diffusivity $\kappa_m$; (ii) strengthening the stabilizing feedback (gain/scaling $\lambda$);
and (iii) designing the actuator placement to improve conditioning.

\paragraph{Fixed--point equation and a universal contraction criterion.}
Let $\{\varphi_k\}_{k\ge 1}$ be the Neumann eigenfunctions of $-\Delta$ in $\Omega$ with eigenvalues
$0=\lambda_1\le\lambda_2\le\cdots$ and set
\[
X_N:=\mathrm{span}\{\varphi_1,\dots,\varphi_N\},\qquad P_N:H\to X_N \text{ the }H\text{--orthogonal projector},\qquad Q_N:=I-P_N .
\]
In the fixed--point construction of Subsection~\ref{sec:fixed-point-bias} (notation as there), \color{black}{we use $\mathcal X:=D(\mathcal A)'$ for the distributional forcing space and $R_{\rm cl}$ for the stationary solution operator on the relevant forcing range.}
The bias operator $T_N:X_N\to X_N$ can be written as
\begin{equation}
\label{eq:TN-abstract}
T_N \;=\; P_N R_{\mathrm{cl}}\, Q_N\, B\,U_N ,
\end{equation}
where:
\begin{itemize}
\item $A y := \kappa_m \Delta y$ (Neumann boundary condition) and \color{black}{$A_{\mathrm{cl}}$ denotes the closed-loop dynamics (or its operator realization when this is used below)};
\item $S_{\mathrm{cl}}(t)=e^{tA_{\mathrm{cl}}}$ is exponentially stable in $V'$ (see Theorem~\ref{thm:main-stabilization});
\item \color{black}{$R_{\mathrm{cl}}$ is the stationary solution map introduced in Subsection~\ref{sec:fixed-point-bias}; the closed-loop semigroup extends to the forcing space and the integral representation is available, $R_{\mathrm{cl}}=\int_0^\infty S_{\mathrm{cl}}(t)\,dt=-A_{\mathrm{cl}}^{-1}$};
\item $B:\mathbb{R}^M\to\mathcal{X}$ is the actuator operator (Dirac point actuators, interpreted in the same dual/extrapolation space $\mathcal{X}$ as in Subsection~\ref{sec:fixed-point-bias});
\item $U_N:X_N\to\mathbb{R}^M$ is the (linear) ``feedforward selector'' used to enforce the cancellation property
\begin{equation}
\label{eq:cancellation-property}
P_N \bigl(A_0 y + B U_N y\bigr)=0,\qquad \forall\,y\in X_N,
\end{equation}
so that the residual forcing generated by the feedforward acts only on the tail.
\end{itemize}
The fixed--point corrected reference $y^\star\in X_N$ is defined by
\begin{equation}
\label{eq:fp-linear-again}
(I+T_N)\,y^\star = y_r \qquad \text{in }X_N.
\end{equation}

\medskip

\noindent
A sufficient condition for unique solvability of \eqref{eq:fp-linear-again} by contraction is $\|T_N\|<1$.
To expose the three mechanisms, we separate the estimate into (i) \emph{low$\leftarrow$high leakage} and
(ii) \emph{conditioning of the feedforward}:
\begin{equation}
\label{eq:TN-factorization}
\|T_N\|_{\mathcal{L}(X_N)}
\;\le\;
\underbrace{\color{black}{\|P_N R_{\mathrm{cl}}Q_N\|_{\mathcal{L}(\mathcal{X},V')}}}_{=: \,L_N}
\;\cdot\;
\|B\|_{\mathcal{L}(\mathbb{R}^M,\mathcal{X})}
\;\cdot\;
\|U_N\|_{\mathcal{L}(X_N,\mathbb{R}^M)}.
\end{equation}

\begin{proposition}[Universal contraction criterion]
\label{prop:universal-contraction}
If
\begin{equation}
\label{eq:universal-contraction}
L_N\,\|B\|\,\|U_N\|<1,
\end{equation}
then $I+T_N$ is invertible on $X_N$, \eqref{eq:fp-linear-again} admits a unique solution $y^\star\in X_N$, and the Picard
iteration $y^{(\ell+1)} := y_r - T_N y^{(\ell)}$ converges geometrically to $y^\star$.
Moreover,
\[
\|y^\star\|_{X_N}\le \frac{1}{1-\|T_N\|}\,\|y_r\|_{X_N}.
\]
\end{proposition}

\begin{proof}
This is immediate from $\|T_N\|<1$ and the Neumann series for $(I+T_N)^{-1}$.
\end{proof}
\noindent
In the remainder of this subsection, we present three independent ways to enforce \eqref{eq:universal-contraction}.
Mechanisms (A) and (B) reduce the leakage factor $L_N$, while mechanism (C) reduces the conditioning factor $\|U_N\|$.

\paragraph{(A) Large diffusivity $\kappa_m$ reduces low$\leftarrow$high leakage.}
\color{black}{The closed-loop dynamics admit the corresponding operator realization on $H$,} we decompose it with respect to $H=X_N\oplus X_N^\perp$:
\[
A_{\mathrm{cl}}
=
\begin{pmatrix}
A_{11} & A_{12}\\
A_{21} & A_{22}
\end{pmatrix},
\qquad
A_{ij}:=P_i A_{\mathrm{cl}} P_j,\;\; P_1:=P_N,\;\; P_2:=Q_N,
\]
and write
\[
A_{22} = \kappa_m\,Q_N\Delta Q_N + E_{22},\qquad E_{22}:=Q_N\bigl(A_{\mathrm{cl}}-\kappa_m\Delta\bigr)Q_N .
\]
Set $\beta:=\|E_{22}\|_{\mathcal{L}(X_N^\perp)}$ (finite because the feedback/observation terms in $A_{\mathrm{cl}}$ are bounded on the chosen state space).

\begin{lemma}[Tail resolvent estimate]
\label{lem:tail-resolvent-kappa}
If $\kappa_m\lambda_{N+1}>\beta$, then $A_{22}$ is invertible on $X_N^\perp$ and
\begin{equation}
\label{eq:tail-inverse}
\|A_{22}^{-1}\|_{\mathcal{L}(X_N^\perp)}
\;\le\;
\frac{1}{\kappa_m\lambda_{N+1}-\beta}.
\end{equation}
\end{lemma}

\begin{proof}
On $X_N^\perp$, the spectrum of $Q_N\Delta Q_N$ is contained in $\{-\lambda_k\}_{k\ge N+1}$, hence
$\|(\kappa_m Q_N\Delta Q_N)^{-1}\|\le (\kappa_m\lambda_{N+1})^{-1}$.
Write $A_{22}=\kappa_m Q_N\Delta Q_N + E_{22}$ and invert by a Neumann series when we have \newline
$\|E_{22}(\kappa_m Q_N\Delta Q_N)^{-1}\|<1$, i.e. $\kappa_m\lambda_{N+1}>\beta$.
\end{proof}
\noindent
Assume in addition that the Schur complement on $X_N$,
\[
S:=A_{11}-A_{12}A_{22}^{-1}A_{21},
\]
is invertible and satisfies $\|S^{-1}\|\le C_S$ (typically guaranteed by exponential stability on $X_N$ of the chosen feedback design).
Then the $(1,2)$--block of $A_{\mathrm{cl}}^{-1}$ satisfies the standard identity
\[
P_N A_{\mathrm{cl}}^{-1} Q_N \;=\; -S^{-1}A_{12}A_{22}^{-1}.
\]
Since $R_{\mathrm{cl}}=-A_{\mathrm{cl}}^{-1}$, we obtain:

\begin{proposition}[Leakage bound improved by large $\kappa_m$]
\label{prop:leakage-kappa}
Assume $\kappa_m\lambda_{N+1}>\beta$ and $\|S^{-1}\|\le C_S$. Then
\begin{equation}
\label{eq:leakage-kappa}
L_N=\|P_N R_{\mathrm{cl}}Q_N\|
\;\le\;
\frac{C_S\,\|A_{12}\|}{\kappa_m\lambda_{N+1}-\beta}.
\end{equation}
Consequently, the contraction condition \eqref{eq:universal-contraction} holds whenever
\begin{equation}
\label{eq:kappa-threshold-final}
\frac{C_S\,\|A_{12}\|}{\kappa_m\lambda_{N+1}-\beta}\,\|B\|\,\|U_N\|<1.
\end{equation}
\end{proposition}

\begin{proof}
Combine the block identity with Lemma~\ref{lem:tail-resolvent-kappa} and \eqref{eq:TN-factorization}.
\end{proof}

\paragraph{(B) Stronger stabilizing feedback (gain $\lambda$) reduces $R_{\mathrm{cl}}$ and hence $L_N$.}
A complementary mechanism is to increase the closed--loop decay rate by strengthening the feedback.
Assume the closed--loop semigroup satisfies the exponential stability estimate
\begin{equation}
\label{eq:exp-stab}
\|S_{\mathrm{cl}}(t)\|_{\mathcal{L}(\mathcal{X})} \le M e^{-\alpha t},\qquad t\ge 0,
\end{equation}
for some $M\ge 1$ and $\alpha>0$. Then
\begin{equation}
\label{eq:Rcl-bound-decay}
\|R_{\mathrm{cl}}\|
\;=\;
\Bigl\|\int_0^\infty S_{\mathrm{cl}}(t)\,dt\Bigr\}
\;\le\;
\int_0^\infty \|S_{\mathrm{cl}}(t)\|\,dt
\;\le\;
\frac{M}{\alpha}.
\end{equation}
Since $\|P_N\|,\|Q_N\|\le 1$ on $H$ (and remain bounded in equivalent norms on $X_N$), this implies the crude but robust bound
\begin{equation}
\label{eq:leakage-alpha}
L_N=\|P_NR_{\mathrm{cl}}Q_N\|
\;\le\;
\|R_{\mathrm{cl}}\|
\;\le\;
\frac{M}{\alpha}.
\end{equation}

\begin{proposition}[Contraction criterion in terms of the closed--loop decay rate]
\label{prop:alpha-contraction}
Assume \eqref{eq:exp-stab}. If
\begin{equation}
\label{eq:alpha-contraction}
\frac{M}{\alpha}\,\|B\|\,\|U_N\|<1,
\end{equation}
then the fixed--point problem \eqref{eq:fp-linear-again} has a unique solution and the Picard iteration converges.
\end{proposition}

\begin{proof}
Combine \eqref{eq:leakage-alpha} with Proposition~\ref{prop:universal-contraction}.
\end{proof}

\begin{remark}[How $\lambda$ enters $\alpha$]
In many finite-dimensional designs (pole placement / Riccati design applied to the truncated model on $X_N$),
one can enforce a prescribed decay margin on $X_N$:
\[
\Re\,\sigma\!\bigl(A_{11}(\lambda)\bigr)\le -\alpha_{\mathrm{low}}(\lambda),
\]
where $\alpha_{\mathrm{low}}(\lambda)$ increases when the feedback scaling $\lambda$ increases, provided the pair
$(A_{11},B_{11})$ is stabilizable on $X_N$.
The global $\alpha$ can be taken as $\alpha=\min\{\alpha_{\mathrm{low}}(\lambda),\,\kappa_m\lambda_{N+1}-\beta\}$,
up to a multiplicative constant $M$ (depending on the chosen norm). This makes \eqref{eq:alpha-contraction}
a practical design criterion: increase $\lambda$ until $\alpha$ is large enough.
\end{remark}

\paragraph{(C) Placement design reduces $\|U_N\|$ via conditioning of the sampling matrix.}
The operator $U_N$ is determined by the cancellation requirement \eqref{eq:cancellation-property}.
For point actuators, the relevant finite-dimensional matrix is the sampling (or ``moment'') matrix
\begin{equation}
\label{eq:CN-def}
C_N \;:=\; \bigl(\varphi_k(z_j)\bigr)_{k=1,\dots,N}^{j=1,\dots,M}\in\mathbb{R}^{N\times M},
\end{equation}
where $\{z_j\}_{j=1}^M\subset\Omega$ are the actuator centers. In typical constructions, $U_N$ is chosen as a right-inverse
of $C_N$ (possibly composed with fixed diagonal scalings depending on the exact definition of $B$ and the chosen basis of $X_N$).
The key quantity is the smallest singular value $\sigma_{\min}(C_N)$.

\begin{lemma}[Right-inverse bound]
\label{lem:right-inverse-bound}
Assume $\mathrm{rank}(C_N)=N$ (in particular $M\ge N$). Let $C_N^\dagger$ denote the Moore--Penrose pseudoinverse.
Then $C_N^\dagger$ is a bounded right-inverse, $C_N C_N^\dagger=I_N$, and
\begin{equation}
\label{eq:pseudoinverse-norm}
\|C_N^\dagger\|_{\mathcal{L}(\mathbb{R}^N,\mathbb{R}^M)} = \frac{1}{\sigma_{\min}(C_N)}.
\end{equation}
Consequently, if $U_N$ is chosen proportional to $C_N^\dagger$ (as in the minimum-norm cancellation choice), then
\begin{equation}
\label{eq:UN-bound-smin}
\|U_N\|\;\lesssim\;\frac{1}{\sigma_{\min}(C_N)}.
\end{equation}
\end{lemma}

\begin{proof}
Standard linear algebra: $\|C_N^\dagger\|=1/\sigma_{\min}(C_N)$ for full row rank matrices.
\end{proof}

\begin{proposition}[Contraction via placement conditioning]
\label{prop:placement-contraction}
Assume $\mathrm{rank}(C_N)=N$ and choose $U_N$ according to the minimum-norm cancellation rule so that
\eqref{eq:UN-bound-smin} holds. Then
\begin{equation}
\label{eq:placement-criterion}
\|T_N\| \;\le\; L_N\,\|B\|\,\|U_N\|
\;\lesssim\;
L_N\,\|B\|\,\frac{1}{\sigma_{\min}(C_N)}.
\end{equation}
In particular, a sufficient condition for solvability of \eqref{eq:fp-linear-again} is
\begin{equation}
\label{eq:smin-threshold}
\sigma_{\min}(C_N)\;\gtrsim\; L_N\,\|B\|.
\end{equation}
\end{proposition}

\begin{proof}
Combine \eqref{eq:TN-factorization} with Lemma~\ref{lem:right-inverse-bound}.
\end{proof}

\begin{remark}[Design of placement]
For fixed $N$, the full-row-rank condition $\mathrm{rank}(C_N)=N$ is generic with respect to the point configuration
$\{z_j\}$ (it fails only on a thin set), and oversampling ($M>N$) typically improves $\sigma_{\min}(C_N)$ and robustness.
A practical design principle is to choose $\{z_j\}$ to maximize $\sigma_{\min}(C_N)$ (or, equivalently, minimize $\|C_N^\dagger\|$),
e.g. by greedy selection from a candidate mesh or by D-optimal design maximizing $\det(C_NC_N^\top)$.
\end{remark}

\paragraph{Combined sufficient condition (three options together).}
We summarize the three mechanisms in a single criterion. Combine the $\kappa_m$--driven leakage estimate
\eqref{eq:leakage-kappa} (mechanism (A)) or the decay-rate estimate \eqref{eq:leakage-alpha} (mechanism (B))
with the placement estimate \eqref{eq:UN-bound-smin} (mechanism (C)). For instance, using \eqref{eq:leakage-kappa} and
Lemma~\ref{lem:right-inverse-bound}, a sufficient contraction condition is
\[
\frac{C_S\,\|A_{12}\|}{\kappa_m\lambda_{N+1}-\beta}\;\|B\|\;\frac{1}{\sigma_{\min}(C_N)} \;<\; 1,
\]
which can be enforced by: increasing $\kappa_m$ (stronger tail dissipation), increasing the feedback gain $\lambda$ so that
$C_S$ and/or $\beta$ improve (stronger closed-loop stability), and/or improving the actuator placement to increase $\sigma_{\min}(C_N)$.

\section*{Acknowledgements}
\noindent The authors used ChatGPT Plus (OpenAI) as a language and editorial assistance tool for improving the clarity, organization, and presentation of the manuscript. All mathematical arguments, results, and conclusions were developed and verified by the authors.

\end{document}